\documentclass[leqno]{amsart}
\usepackage{amsmath,amssymb,amsfonts,latexsym,stmaryrd}
\usepackage{graphicx,float,epsfig} 
\usepackage{mathrsfs} 
\usepackage{float} 
\usepackage{color}
\usepackage{setspace} 

\DeclareGraphicsExtensions{ .eps,.ps} \usepackage{subfig}
\usepackage{multirow}
\usepackage{url}
\usepackage{diagbox}
\usepackage[linesnumbered,algoruled,boxed]{algorithm2e}
\usepackage[colorlinks=false, linkcolor=blue,  citecolor=OliveGreen,  pdfstartview={}{}]{hyperref} 
\usepackage[dvipsnames]{xcolor}

\usepackage{svg}
\usepackage{amsmath}

\def\O{\Omega}

\def\l{\lambda}

\newtheorem{remark}{Remark}[section]
\newtheorem{lemma}{Lemma}[section]
\newtheorem{theorem}{Theorem}[section]
\newtheorem{prop}{Proposition}[section]
\newtheorem{assumption}{Assumption}[section]
\newcommand{\jump}[1]{\Big\llbracket #1 \Big\rrbracket}
\newcommand{\jumpp}[1]{\llbracket #1 \rrbracket}

\usepackage{booktabs}
\usepackage{float}

\newcommand\bu{\boldsymbol{u}}
\newcommand\bv{\boldsymbol{v}}
\newcommand\bw{\boldsymbol{w}}

\newcommand\bn{\boldsymbol{n}}

\newcommand\curl{\textbf{\text{curl}\,}}

\newcommand\bxi{\boldsymbol{\xi}}

\def\CT{{\mathcal T}}

\newcommand\bsig{\boldsymbol{\sigma}}
\newcommand\btau{\boldsymbol{\tau}}
\newcommand\bPi{\boldsymbol{\Pi}}

\newcommand\R{\mathbb{R}}

\renewcommand\H{\mathrm{H}}
\renewcommand\L{\mathrm{L}}

\renewcommand\O{\Omega}

\newcommand\bdiv{\mathop{\mathbf{div}}\nolimits}

\renewcommand\div{\mathop{\mathrm{div}}\nolimits}
\newcommand\rot{\mathop{\mathrm{rot}}\nolimits}

\newcommand\tr{\mathop{\mathrm{tr}}\nolimits}

\newcommand\LO{\L^2(\O)}

\newcommand\ws{\widehat{s}}

\newcommand\err{\texttt{err}}
\newcommand\eff{\texttt{eff}}

\newcommand{\vertiii}[1]{{\left\vert\kern-0.25ex\left\vert\kern-0.25ex\left\vert #1 
    \right\vert\kern-0.25ex\right\vert\kern-0.25ex\right\vert}}

\begin{document}

\title[AFEM for the mixed  elasticity eigenproblem]
{A posteriori analysis for a mixed FEM discretization of the linear elasticity spectral problem}


\author{Felipe Lepe}
\address{GIMNAP-Departamento de Matem\'atica, Universidad del B\'io - B\'io, Casilla 5-C, Concepci\'on, Chile.}
\email{flepe@ubiobio.cl}
\thanks{The first author was partially supported by DIUBB through project 2120173 GI/C Universidad del B\'io-B\'io and 
ANID-Chile through FONDECYT project 11200529 (Chile).}

\author{Gonzalo Rivera}
\address{Departamento de Ciencias Exactas,
Universidad de Los Lagos, Casilla 933, Osorno, Chile.}
\email{gonzalo.rivera@ulagos.cl}

\author{Jesus Vellojin}
\address{GIMNAP-Departamento de Matem\'atica, Universidad del B\'io - B\'io, Casilla 5-C,  Concepci\'on, Chile .}
\email{jesus.vellojinm@usm.cl}




\begin{abstract}
In this paper we analyze a posteriori error estimates for a mixed formulation of the linear elasticity eigenvalue problem. A posteriori estimators for the  nearly and perfectly compressible elasticity spectral problems are proposed. With a post-process argument, we are able to prove reliability and efficiency
for the proposed estimators. The numerical method is based in Raviart-Thomas elements to approximate the pseudostress and piecewise polynomials for the displacement. We illustrate our results with numerical tests. 
\end{abstract}

\maketitle

\section{Introduction}\label{sec:intro}

In several applications of engineering sciences or physics, there exist problems where an accurate knowledge of the eigenvalues and eigenfunctions 
is  needed in order to analyze the stability and response of certain mechanical systems.
The different configurations in which such systems can be formulated, depend on physical features as 
material properties, contact with other structures or devices, just to mention a few, and geometrical features, since in real applications, elastic structures
can be used in locations that might be  convex, non convex, curved, fractured, etc.  Is this fact that leads to 
develop adaptive strategies for numerical methods  in order to recover the optimal order of convergence  for eigenvalue
problems in partial differential equations.

The literature related to adaptive strategies for the elasticity equations is abundant for the load problem, where different methods, formulations and techniques have been
well developed. On this subject, we can mention as main references \cite{MR3452773,MR3790080,MR2970742,MR3453481,MR3093586,MR2220917}, whereas for the elasticity spectral problems, the literature available is scarce. In fact, there are three works where a posteriori error  analysis for the elasticity eigenproblem is performed:\cite{MR1946986,MR4279087, MR4050542}.

For mixed eigenvalue problems, adaptive methods are a subject of current study and different techniques have emerged. One of the pioneer results  are contained in the classic article  \cite{MR1722056}, where the authors have proved that the  mixed Laplace eigenvalue problem, that the high order terms that naturally appear in the a posteriori estimators for the eigenvalues, are controlled by considering an auxiliary problem discretized  with 
a non-conforming method which results equivalent with the original one. On the other hand, and with the same aim of the reference previously mentioned, the postprocess technique, well
established in \cite{ MR2754580,MR3047040}, presents a new tool for the control of high order terms in eigenvalue problem. A recent application of this technique can be found in, for instance, on  \cite{MR3712172,MR3918688}.  

The present work is inspired in the mixed formulation proposed in \cite{MR3453481} for the source elasticity problem, where the authors introduce the nonsymmetric pseudoestress tensor as a new unknown, together with the displacement. This pseudostress tensor gives an  alternative way of dealing with dual-mixed variational formulations in continuum mechanics, without the need of imposing neither strong nor weak symmetry of the classic stresses.
It is precisely this tensor that leads to a tensorial formulation for the elasticity equations and, as a natural extension, the spectral elasticity problem can be also considered   as  in \cite{inzunza2021displacementpseudostress}, where only the a priori analysis is performed.
Since the elasticity system depends on the 
Lam\'e constants, often denoted by $\mu$ and $\lambda$, it is well known that when the Poisson ratio is close to $1/2$, numerical locking arises since $\lambda\rightarrow\infty$. This motivates the study of the so-called limit eigenvalue  problem (see \cite{inzunza2021displacementpseudostress, MR3962898} for instance). Hence, it is possible to consider two types of estimators: one for the limit eigenproblem and the other for the standard eigenproblem. Our task is to design a reliable and efficient a posteriori estimator for both problems in two and three dimensional domains and analyze computationally their performance, with the aim of recovering the optimal order of convergence for the eigenvalues and eigenfunctions. 

We remark that, for simplicity, our analysis is devoted to the spectral elasticity problem with only Dirichlet boundary conditions (cf. Section \ref{sec:model}), since the mixed boundary conditions on the domain implies the imposition of the normal component of the pseudotress tensor 
on the system and hence, on the Hilbert space in which the solution lies, leading to an analysis with other difficulties, like the regularity of the eigenfunctions for instance, that we will perform in other paper according to our research program.

The paper is organized as follows: in Section \ref{sec:model} we present the elasticity eigenvalue problem and the mixed formulation of interest.
We summarize some results related to its stability and spectral characterization. In Section \ref{sec:FEMMM} we introduce the discrete mixed eigenvalue problem, particularly the FEM spaces for the approximation and the post-process technique. The core of our paper is section \ref{sec:apost},
where the local and global indicators are presented for the standard and limit eigenproblems. Reliability and efficiency analyses for the proposed estimators  are performed. Finally in section \ref{sec:numerics} we report
some numerical tests in order to analyze the performance of the error estimators in two and three dimensions.

We end this section with some notations that will use below. Given $n\in\{ 2,3\}$, we denote $\mathbb{R}^{n\times n}$ the space of vectors and tensors of order $n$ with entries in $\mathbb{R}$, and $\mathbb{I}$ is the identity matrix of $\mathbb{R}^{n\times n}$. Given any $\boldsymbol{\tau}:=(\tau_{ij})$ and $\boldsymbol{\sigma}:=(\sigma_{ij})\in \mathbb{R}^{n\times n}$, we write
\[
	\boldsymbol{\tau}^{\texttt{t}}:=(\tau_{ji}), \quad \tr(\boldsymbol{\tau}):=\sum_{i=1}^{n}\tau_{ii}, \quad \mbox{and} \quad \boldsymbol{\tau:\sigma}:=\sum_{i,j=1}^{n} \tau_{ij}\,\sigma_{ij}, 
\]
to refer to the transpose, the trace and the tensorial product between $\boldsymbol{\tau}$ and $\boldsymbol{\sigma}$ respectively. 

For $s\geq 0$, we denote as $\| \cdot \|_{s, \O}$ the norm of the Sobolev space $\H^{s}(\O)$ or
$\mathbb{H}^s(\O):=[\H^{s}(\O)]^{n\times n}$ with $n\in\{2,3\}$  for scalar  and tensorial fields, respectively, with the convention $\H^0(\O):=\LO$ and $\mathbb{H}^0(\O):=\mathbb{L}^2(\O)$. Furthermore, with $\div$ denoting the usual divergence operator, we define the Hilbert space
\[
	\H(\div, \O):=\{ \boldsymbol{f} \in \L^{2}(\O) \,:\, \div(\boldsymbol{f}) \in \L^{2}(\O) \},
\] 
equipped with the norm $ \| \boldsymbol{f} \|_{\div, \O}^{2}:= \|\boldsymbol{\tau} \|_{0,\O}^{2} + \|\div(\boldsymbol{f}) \|_{0,\O}^{2} $, and
the space 
\begin{equation*}
\mathbb{H}(\curl, \Omega):=\{\bw\in\mathbb{L}^2(\O):\,\curl\bw\in\mathbb{L}^2(\O)\},
\end{equation*}
that we endow with the norm  $ \| \boldsymbol{w} \|_{\curl, \O}^{2}:= \|\boldsymbol{w} \|_{0,\O}^{2} + \|\curl(\boldsymbol{w}) \|_{0,\O}^{2} $.

The space of matrix valued functions whose rows belong to $\H(\div, \O)$ will be denoted $\mathbb{H}(\bdiv, \O)$ where $\bdiv$ stands for the action of $\div$ along each row of a tensor.

Finally, we use $C$ with or without subscripts, bar, tildes or hat, to denote  generic constants independent of the discretization parameter, which may take different values at different places.  

%
\section{The linear elasticity eigenvalue problem}
\label{sec:model}
The elasticity eigenvalue problem of our interest is the following 
\begin{equation*}\label{def:elast_system_reduced}
\left\{
\begin{array}{rcll}
\mu\Delta\bu+(\lambda+\mu)\nabla\div\bu & = & -\kappa\bu &  \text{ in } \quad \Omega, \\
\bu & = & \mathbf{0} & \text{ on } \quad \partial\Omega,
\end{array}
\right.
\end{equation*}
where the Cauchy stress tensor $\bsig$ is such that
\begin{equation*}
\bdiv\bsig=2\mu\bdiv\boldsymbol{\varepsilon}(\bu)+\lambda\nabla\div\bu=\mu\Delta\bu+(\lambda+\mu)\nabla\div\bu,
\end{equation*}
and $\boldsymbol{\varepsilon}(\bu)$ is the tensor of small deformations defined by $\boldsymbol{\varepsilon}(\bu)=\frac{1}{2}(\nabla\bu+(\nabla\bu)^{\texttt{t}})$. Now, with  the so-called pseudostress tensor, defined by
$
\boldsymbol{\rho}:=\mu\nabla\bu+(\lambda+\mu)\tr(\nabla\bu)\mathbb{I},
$
we obtain  the following  system
\begin{equation*}\label{def:elast_system_rho_1}
\left\{
\begin{array}{rcll}
\boldsymbol{\rho} & = &\mu\nabla\bu+(\lambda+\mu)\tr(\nabla\bu)\mathbb{I}&  \text{ in } \quad \Omega, \\
\div\boldsymbol{\rho} & = & -\kappa\bu & \text{ in } \quad \Omega, \\
\bu & = & \mathbf{0} & \text{ on } \quad \partial\Omega,
\end{array}
\right.
\end{equation*}
which we rewritte as follows
\begin{equation}\label{def:elast_system_rho}
\left\{
\begin{array}{rcll}
\displaystyle\frac{1}{\mu}\left\{\boldsymbol{\rho}-\frac{\lambda+\mu}{n\lambda+(n+1)\mu}\tr(\boldsymbol{\rho})\mathbb{I} \right\}& = &\nabla\bu&  \text{ in } \quad \Omega, \\
\div\boldsymbol{\rho} & = & -\kappa\bu & \text{ in } \quad \Omega, \\
\bu & = & \mathbf{0} & \text{ on } \quad \partial\Omega.
\end{array}
\right.
\end{equation}

Multiplying the above system with suitable tests functions, integrating by parts and using the boundary condition, we obtain the following variational formulation: Find $\kappa\in\mathbb{R}$ and $\boldsymbol{0}\neq (\boldsymbol{\rho},\bu)\in\mathbb{H}\times \mathbf{Q}$, such that

\begin{equation}\label{def:spectral_1}
\left\{
\begin{array}{rcll}
a(\boldsymbol{\rho},\btau)+b(\btau,\bu) & = &0&  \forall\btau\in\mathbb{H}, \\
b(\boldsymbol{\rho},\bv)& = & -\kappa(\bu,\bv)_{0,\O} &  \forall\bv\in \mathbf{Q},
\end{array}
\right.
\end{equation}
where $\mathbb{H}:=\mathbb{H}(\bdiv,\O)$ and $\mathbf{Q}:=\L^2(\O)^n$ and the bilinear forms $a:\mathbb{H}\times\mathbb{H}\rightarrow \mathbb{R}$ and $b:\mathbb{H}\times \mathbf{Q}\rightarrow\mathbb{R}$ are defined by
\begin{equation}
\label{eq:a_original}
\displaystyle a(\bxi,\btau):=\frac{1}{\mu}\int_{\Omega}\bxi:\btau-\frac{\lambda+\mu}{\mu(n\lambda+(n+1)\mu)}\int_{\O}\tr(\bxi)\tr(\btau)\quad\forall\bxi,\btau\in\mathbb{H},
\end{equation}
and 
\begin{equation*}
b(\btau,\bv):=\int_{\O}\bv\cdot\bdiv\btau\quad\forall\btau\in\mathbb{H},\,\,\forall\bv\in \mathbf{Q}.
\end{equation*}

For $\btau\in\mathbb{H}$ we define its associated deviator tensor by $\btau^{\texttt{d}}:=\btau-\frac{1}{n}\tr(\btau)\mathbb{I}$, which allows us to redefine $a(\cdot,\cdot)$ as follows
\begin{equation}
\label{eq:identity_a}
\displaystyle a(\bxi,\btau):=\frac{1}{\mu}\int_{\Omega}\bxi^{\texttt{d}}:\btau^{\texttt{d}}+\frac{1}{n(n\lambda+(n+1)\mu)}\int_{\O}\tr(\bxi)\tr(\btau)\quad\forall\bxi,\btau\in\mathbb{H}.
\end{equation}

With the purpose of establish the well posedness of the mixed formulation \eqref{def:spectral_1}, we introduce the following decomposition $\mathbb{H}:=\mathbb{H}_0\oplus \R \mathbb{I}$ where
\begin{equation*}
\mathbb{H}_0:=\left\{\btau\in\mathbb{H}\,:\,\int_{\O}\tr(\btau)=0\right\}.
\end{equation*}
Note that for any $\bxi\in\mathbb{H}$ there exists a unique $\bxi_{0}\in \mathbb{H}_0$ and $d:=\dfrac{1}{n|\O|}\displaystyle\int_{\O}\tr(\bxi)\in\R$, such that the decomposition for $\bxi$ holds.

The following lemma guarantees that the test space can also be restricted to $\mathbb{H}_0$ 
\begin{lemma}
Any solution of \eqref{def:spectral_1} with $\boldsymbol{\rho}\in \mathbb{H}_{0}$ is also solution of the problem: Find $\kappa\in\mathbb{R}$ and $\boldsymbol{0}\neq (\boldsymbol{\rho}_0,\bu_0)\in\mathbb{H}_{0}\times \mathbf{Q}$, such that

\begin{equation}\label{def:spectral_H0}
\left\{
\begin{array}{rcll}
a(\boldsymbol{\rho}_{0},\btau)+b(\btau,\bu_0) & = &0&  \forall\btau\in\mathbb{H}_{0}, \\
b(\boldsymbol{\rho}_{0},\bv)& = & -\kappa(\bu_0,\bv)_{0,\O} &  \forall\bv\in \mathbf{Q}.
\end{array}
\right.
\end{equation}
Conversely, any solution of \eqref{def:spectral_H0} is also a solution of \eqref{def:spectral_1}.
\end{lemma}

Let us consider the source problem associated to \eqref{def:spectral_H0}: given $\boldsymbol{f}\in\mathbf{Q}$, find $(\widehat{\boldsymbol{\rho}}_0,\widehat{\bu}_0)\in\mathbb{H}_0\times\mathbf{Q}$, such that 
\begin{equation}\label{def:ssource_H0}
\left\{
\begin{array}{rcll}
a(\widehat{\boldsymbol{\rho}}_{0},\btau)+b(\btau,\widehat{\bu}_0) & = &0&  \forall\btau\in\mathbb{H}_{0}, \\
b(\widehat{\boldsymbol{\rho}}_{0},\bv)& = & -(\boldsymbol{f},\bv)_{0,\O} &  \forall\bv\in \mathbf{Q}.
\end{array}
\right.
\end{equation}

From the proof of \cite[Lemma 4.1]{MR3453481},  an important consequence of the well posedness of \eqref{def:ssource_H0}, is  that there exists a constant $C>0$ such that the pair $(\widehat{\boldsymbol{\rho}}_0,\widehat{\bu}_0)$ satisfies $\|\widehat{\boldsymbol{\rho}}_{0}\|_{\bdiv,\O}+\|\widehat{\bu}\|_{0,\O}\leq C\|\boldsymbol{f}\|_{0,\O}$ (see \cite[Theorem 2.1]{MR3453481} ). Hence, if $\mathcal{A}:\mathbb{H}_0\times\mathbf{Q}\rightarrow(\mathbb{H}_0\times\mathbf{Q})' $ is the linear operator associated to the left hand side of \eqref{def:ssource_H0}, it is possible to prove that $\mathcal{A}$ is an isomorphism that satisfies $\|\mathcal{A}(\btau,\bv)\|_{(\mathbb{H}_0\times\mathbf{Q})'}\geq C\|(\btau,\boldsymbol{v})\|_{\mathbb{H}_0\times\mathbf{Q}}$, for all $(\btau,\boldsymbol{v})\in\mathbb{H}_0\times\mathbf{Q}$, which is equivalent to the following inf-sup condition
\begin{equation}
\label{eq:complete_infsup}
\|(\boldsymbol{\tau}\hspace*{-0.045cm}, \bv)\|_{\mathbb{H}_{0}\times \mathbf{Q}}\leq C\left(\hspace*{-0.1cm}\displaystyle\sup_{\underset{(\boldsymbol{\xi},\boldsymbol{w})\neq \boldsymbol{0}}{(\boldsymbol{\xi},\boldsymbol{w})\in \mathbb{H}_{0}\times\mathbf{Q}}}\frac{a(\boldsymbol{\tau},\boldsymbol{\xi})\hspace*{-0.045cm}+\hspace*{-0.045cm}b(\boldsymbol{\xi},\bv)\hspace*{-0.045cm}+\hspace*{-0.045cm}b(\btau,\boldsymbol{w})}{\|(\boldsymbol{\xi},\boldsymbol{w})\|_{\mathbb{H}_{0}\times \mathbf{Q}}}\right),
\end{equation}
where $C$ is a positive constant.

%

We end this section with the following regularity result (see \cite{inzunza2021displacementpseudostress} for instance).
\begin{lemma}[Regularity of the eigenfunctions]
\label{lmm:add_eigen}
The solutions $(\kappa,\boldsymbol{\rho},\bu)$ of the problem above correspond, in one hand, to a sequence of positive finite-multiplicity eigenvalues $\{\kappa_i\}_{i\in\mathbb{N}}$ such that $\kappa_i\rightarrow\infty$, whereas the pair  $(\boldsymbol{\rho},\bu)\in ^^s(\Omega)\times\H^{1+s}(\Omega)^n$ for all $s\in (0,\ws)$, where $0<\ws\leq1$ (see \cite{MR961439,MR840970} for instance). Also, there exists a constant $\widehat{C}>0$ which in principle depends on $\lambda$, such that
\begin{equation*}
\|\bu\|_{1+s,\O}\leq \widehat{C}\|\bu\|_{0,\O}.
\end{equation*}

\end{lemma} 

We mention that the dependency of the constants in the regularity exponents and boundedness on $\lambda$ is not completely evident. This has been observed in \cite{inzunza2021displacementpseudostress,MR3962898} when the numerical tests are performed.
This motivates us to consider  the following assumption along our paper:
\begin{assumption}
Constants $\widehat{s}$ and $\widehat{C}$ in Lemma \ref{lmm:add_eigen} are independent of $\lambda$.
\end{assumption}
\section{The discrete eigenvalue problem}
\label{sec:FEMMM}
\subsection{The finite element spaces}
Given an integer $\ell\geq 0$ and a subset $D$ of $\mathbb{R}^n$, we denote by $\mathrm{P}_\ell(S)$ the space of polynomials of degree at most $\ell$ defined in $D$. We mention that, for tensorial fields we will define $\mathbb{P}_\ell(D):=[\mathrm{P}_\ell(D)]^{n\times n}$ and for vector fields $\mathbf{P}_\ell(D):=[\mathrm{P}_\ell(D)]^n$. With these ingredients at hand, for $k=0$ we define the local Raviart-Thomas space of the lowest order,
 as follows  (see \cite{MR3097958})
 \begin{equation*}
 \mathbf{RT}_0(T)=[\mathbf{P}_0(T)]\oplus \mathrm{P}_0(T)\boldsymbol{x},
 \end{equation*}
 where $\boldsymbol{x}\in\mathbb{R}^n$. With this local space, we define the global Raviart-Thomas space, which we denote by $\mathbb{RT}_0(\CT_h)$, as follows
 \begin{equation*}
 \mathbb{RT}_0(\CT_h):=\{\btau\in\mathbb{H}\,:\,(\tau_{i1},\cdots,\tau_{in})^{\texttt{t}}\in\mathbf{RT}_0(T)\,\,\forall i\in\{1,\ldots,n\},\,\,\forall T\in\CT_h\},
 \end{equation*}
 and we introduce the global space of piecewise polynomials of degree $\leq k$ defined by
 \begin{equation*}
 P_k(\CT_h):=\{v\in \L^2(\O)\,:\, v|_T\in \mathrm{P}_k(T)\,\,\,\,\forall T\in\CT_h\}.
 \end{equation*}
 
 Also, we define
 \begin{equation*}
 \mathbb{H}_{h,0}:=\mathbb{RT}_0(\CT_h)\cap\mathbb{H}_0(\bdiv;\O)=\left\{ \btau_h\in\mathbb{RT}_k(\CT_h)\,\,:\,\,\int_{\O}\tr(\btau_h)=0  \right\},
 \end{equation*}
 and $\mathbf{Q}_h:=\mathbf{P}_0(\CT_h)$. 
 
Now we recall some well known approximation properties for the spaces defined above (see \cite{MR2009375} for instance).  Let $\bPi_h^0:\mathbb{H}^t (\O)\rightarrow \mathbb{RT}_0(\CT_h)$ be the Raviart-Thomas interpolation operator. For $t\in (0,1/2]$ and $\btau\in\mathbb{H}^t(\O)\cap\mathbb{H}(\bdiv;\O)$ the following error estimate holds true
 \begin{equation} \label{daniel1}
 \|\btau-\bPi_h^0\btau\|_{0,\O}\leq Ch^t \big(\|\btau\|_{t,\O}+\|\bdiv\btau\|_{0,\O}\big).
 \end{equation}
 
 Also, for $\btau\in\mathbb{H}^t(\O)$ with $t>1/2$, there holds
 \begin{equation}\label{daniel2}
 \|\btau-\bPi_h^0\btau\|_{0,\O}\leq Ch^{\min\{t,1\}} |\btau|_{t,\O}.
 \end{equation} 
 
 Let $\mathcal{P}_h^0:\L^2(\O)^n\rightarrow\mathbf{Q}_h$ be the $\L^2(\O)$-orthogonal projector. As a first property, we have the following commutative diagram
 \begin{equation}
 \label{eq:commutative}
 \bdiv(\bPi_h^0\btau)=\mathcal{P}_h^0(\bdiv\btau).
 \end{equation}
 
 If $\bv\in\H^t (\O)^{n}$ with $t>0$, there holds
 \begin{equation}\label{daniel3}
 \|\bv-\mathcal{P}_h^0\bv\|_{0,\O}\leq Ch^{\min\{t,1 \}} |\bv|_{t,\O}.
 \end{equation}
 
 Finally, for each $\btau\in\mathbb{H}^t(\O)$ such that $\bdiv\btau\in\H^t (\O)^{n}$, there holds
 \begin{equation} \label{daniel4}
 \|\bdiv(\btau-\bPi_h^0)\|_{0,\O}\leq Ch^{\min\{t,1\}} |\bdiv\btau|_{t,\O}.
 \end{equation}

 \subsection{The discrete mixed eigenvalue problem}
 \label{sec:fem}
Now we introduce the finite element discretization of \eqref{def:spectral_1}, which  reads as follows:
 Find $\kappa_h\in\mathbb{R}$ and $(\boldsymbol{\rho}_h,\bu_h)\in \mathbb{H}_{h,0}\times \mathbf{Q}_h $ such that
 
 \begin{equation}\label{def:spectral_1h}
\left\{
\begin{array}{rcll}
a(\boldsymbol{\rho}_h,\btau_h)+b(\btau_h,\bu_h) & = &0&  \forall\btau_h\in\mathbb{H}_{h,0}, \\
b(\boldsymbol{\rho}_h,\bv_h)& = & -\kappa_h(\bu_h,\bv_h)_{0,\O} &  \forall\bv_h\in \mathbf{Q}_h.
\end{array}
\right.
\end{equation}

We introduce the discrete kernel of $b(\cdot,\cdot)$ as follows
\begin{equation*}
\mathbb{V}_h:=\{\btau_h\in\mathbb{H}_{0,h}\,:\, \bdiv\btau_h=\boldsymbol{0}\,\,\text{in}\,\O\}\subset\mathbb{V}.
\end{equation*}
Then, since $a(\cdot,\cdot)$ is $\mathbb{V}_h$-elliptic  and the following inf-sup condition holds (see \cite[Lemma 3.1]{MR3453481})
\begin{equation*}
\displaystyle\sup_{\underset{\btau\neq\boldsymbol{0}}{\btau\in\mathbb{H}_{0,h}}}\frac{b(\btau_h,\bv_h)}{\|\btau_h\|_{\bdiv,\O}}\geq\beta\|\bv_h\|_{0,\O}\quad\forall\bv_h\in \mathbf{Q}_h,
\end{equation*}
where $\beta>0$ is independent of $h$. 
%

In what follows, we assume that $\kappa$ is a simple eigenvalue and we normalize $\bu$ so that $\|\bu\|_{0,\O}=1$. Then, for all $\CT_{h}$, there exists a solution $(\kappa_{h},\boldsymbol{\rho}_h,\bu_h)\in\R\times\mathbb{H}_{0,h}\times\mathbf{Q}_h$ of  \eqref{def:spectral_1h}  such that $\kappa_{h}\rightarrow\kappa$ and $\|\bu_{h}\|_{0,\O}=1$.

The following result, summary a priori error estimates for problems  \eqref{def:spectral_H0} and \eqref{def:spectral_1h} are derived from \cite[Theorems 4.1 and 4.2]{inzunza2021displacementpseudostress}.
\begin{lemma}
\label{lema:apriorie}
Let $(\kappa, \boldsymbol{\rho},\bu)$ be a solution of Problem \eqref{def:spectral_H0} with $\|\bu\|_{0,\O}=1$. Then, there exists a solution $(\kappa_h, \boldsymbol{\rho}_h,\bu_h)$ be a solution of Problem \eqref{def:spectral_1h} with $\|\bu_h\|_{0,\O}=1$. Then
\begin{align*}
\|\boldsymbol{\rho}-\boldsymbol{\rho}_h\|_{0,\O}+\|\bu-\bu_h\|_{0,\O}&\leq C h^{s},\\
|\kappa-\kappa_h|&\leq C\left(\|\boldsymbol{\rho}-\boldsymbol{\rho}_h\|_{0,\O}^2+\|\bu-\bu_h\|_{0,\O}^2\right),
\end{align*}
where the constant $C$ is independent of $h$ and $\lambda$.
\end{lemma}

An important result,  consequence of the spurious free feature of the proposed finite element method, states that for $h$ small enough, except 
for $\kappa_h$, the rest of the eigenvalues of \eqref{def:spectral_1h} are well separated from $\kappa$ (see  \cite{MR3647956}).

\begin{prop}\label{separa_eig}
Let us enumerate the eigenvalues of  \eqref{def:spectral_1h} and \eqref{def:spectral_1} in increasing order as follows: $0<\kappa_1\leq\cdots\kappa_i\leq\cdots$ and 
$0<\kappa_{h,1}\leq\cdots\kappa_{h,i}\leq\cdots$. Let us assume  that $\kappa_J$ is a simple eigenvalue of \eqref{def:spectral_1h}. Then, there exists $h_0>0$ such that
\begin{equation*}
|\kappa_J-\kappa_{h,i}|\geq\frac{1}{2}\min_{j\neq J}|\kappa_j-\kappa_J|\quad\forall i\leq \dim\mathbb{H}_h,\,\,i\neq J,\quad \forall h<h_0.
\end{equation*}
\end{prop}

\subsection{Superconvergence and Postprocessing}
In this section we  derive a superconvergence result between the eigenfunction $\bu$ and its finite element  approximation, together with a  postprocess for the aforementioned unknown.   For simplicity, we only concentrate on the simple eigenvalue case along our paper. 

The forthcoming analysis is inspired by \cite{MR3712172,MR3918688,MR3047040}. Consider the following mixed problem: Find $(\widetilde{\boldsymbol{\rho}}_h,\widetilde{\bu}_h)\in\mathbb{H}_{h,0}\times\mathbf{Q}_h$ such that
\begin{equation}\label{eq:3.3a}
\left\{
\begin{array}{rcll}
a(\widetilde{\boldsymbol{\rho}}_h,\btau_h)+b(\btau_h,\widetilde{\bu}_h)&=&0\quad&\forall\,\btau_h\in\mathbb{H}_{h,0}, \\
b(\widetilde{\boldsymbol{\rho}}_h,\bv_h)&=&-\kappa(\widetilde{\bu}_h,\bv_h)_{0,\O}\quad&\forall \,\bv\in\mathbf{Q}_h, 
\end{array}
\right.
\end{equation}
where  the solution $(\widetilde{\boldsymbol{\rho}}_h,\widetilde{\bu}_h)$ of \eqref{eq:3.3a}  is the finite element approximation of  $(\boldsymbol{\rho},\bu)$.

The proofs of the following results are inspired in \cite{MR3918688}, but taking into account the bilinear form $a(\cdot,\cdot)$ defined in \eqref{eq:identity_a}, where all the estimates are independent of the Lam\'e constant $\lambda$. The first auxiliary result shows a higher-order approximation between  $\widetilde{\bu}_{h}$ and $\mathcal{P}_h^0\bu$.
\begin{lemma}
\label{lmm:rodolfo9}
Let $(\kappa, \boldsymbol{\rho},\bu)$ be a solution of Problem \eqref{def:spectral_H0} and $(\widetilde{\boldsymbol{\rho}}_{h},\widetilde{\bu}_{h})$ be a solution of \eqref{eq:3.3a}. Then, there holds 
\begin{equation*}
\|\widetilde{\bu}_h-\mathcal{P}_h^0\bu\|_{0,\O}\leq C h^{s}\|\boldsymbol{\rho}-\widetilde{\boldsymbol{\rho}}_h\|_{\bdiv,\O},
\end{equation*}
where $s\in (0,\widehat{s}\,]$ and the  positive constant $C$ is independent of $h$ and $\lambda$.
\end{lemma}
\begin{proof}
Adapting the proof of Lemma 9 of \cite{MR3918688}, together with Lemma 2.1 and Remark 2.1 of \cite{inzunza2021displacementpseudostress}, it follows that:
\begin{equation*}
\|\widetilde{\bu}_h-\mathcal{P}_h^0\bu\|_{0,\O}\leq C h^{s}\|\boldsymbol{\rho}-\widetilde{\boldsymbol{\rho}}_h\|_{\bdiv,\O}.
\end{equation*}
This concludes the proof.
\end{proof}
 
The following auxiliary result shows that the term $\|\boldsymbol{\rho}-\widetilde{\boldsymbol{\rho}}_h\|_{\bdiv,\O}$ is bounded.
\begin{lemma}
\label{lmm:rodolfo10}
Let $(\kappa, \boldsymbol{\rho},\bu)$ be a solution of Problem \eqref{def:spectral_H0} and $(\widetilde{\boldsymbol{\rho}}_{h},\widetilde{\bu}_{h})$ be a solution of \eqref{eq:3.3a}. Then, there holds
\begin{equation*}
\|\boldsymbol{\rho}-\widetilde{\boldsymbol{\rho}}_h\|_{\bdiv,\O}\leq C\left( \|\boldsymbol{\rho}-\boldsymbol{\rho}_h\|_{0,\O}+\|\bu-\bu_{h}\|_{0,\O}\right),
\end{equation*}
 where the  positive constant $C$  is independent of $h$ and $\lambda$.
\end{lemma}
\begin{proof} The proof follows from \cite[proof of Lemma 10]{MR3918688}, Lemma \ref{lema:apriorie} and   the stability of the  discrete problem (see \cite{inzunza2021displacementpseudostress}).
\end{proof}
The following identity, proved in \cite[Section 4]{MR3918688} for the mixed problem related to the Maxwell's spectral problem, also holds for our equivalent mixed problem for the elasticity spectral formulation
\begin{lemma}
\label{lmm:rodolfo11}
Let $(\kappa,\boldsymbol{\rho},\bu)$ and $(\kappa_{h}, \boldsymbol{\rho}_{h},\bu_{h})$ be  solutions of Problems \eqref{def:spectral_H0} and \ref{def:spectral_1h}, respectively, with $\|\bu\|_{0,\O}=\|\bu_{h}\|_{0,\O}=1$. Then, there holds
\begin{equation*}
\|\mathcal{P}_h^0\bu-\bu_{h}\|_{0,\O}\leq Ch^{s}\left(\|\boldsymbol{\rho}-\boldsymbol{\rho}_h\|_{0,\O}+\|\bu-\bu_h\|_{0,\O}\right),
\end{equation*}
where $s\in (0,\widehat{s}\,]$ and the positive constant $C$ is independent of $h$.
\end{lemma}
\begin{proof}
Let $(\widetilde{\boldsymbol{\rho}}_{h},\widetilde{\bu}_{h})$ be the solution of \eqref{eq:3.3a}. Then, from the triangle inequality we have
\begin{equation*}
\|\mathcal{P}_h^0\bu-\bu_{h}\|_{0,\O}\leq \|\mathcal{P}_h^0\bu-\widetilde{\bu}_{h}\|_{0,\O}+\|\widetilde{\bu}_{h}-\bu_{h}\|_{0,\O}.
\end{equation*} 
Now, adapting the arguments of \cite[Lemma 11]{MR3918688} and  invoking Proposition \ref{separa_eig}, we derive the following estimate
\begin{equation*}
\|\widetilde{\bu}_{h}-\bu_{h}\|_{0,\O}^{2}\leq C\left(\|\widetilde{\bu}_{h}-\mathcal{P}_h^0\bu\|_{0,\O}^{2}+\left[\|\boldsymbol{\rho}-\boldsymbol{\rho}_h\|_{0,\O}^2+\|\bu-\bu_h\|_{0,\O}^2\right]^{2}\right).
\end{equation*}

 Finally, using the above estimates, Lemmas \ref{lmm:rodolfo9} and \ref{lmm:rodolfo10}, we conclude the proof. 
\end{proof}

We define the following finite element subspace 
\begin{equation*}
\mathbf{Y}_h:=\left\{\bv\in [\H^1(\Omega)]^n\,:\,\bv\in \mathbf{P}_1(T), \quad\forall T\in\mathcal{T}_h\right\}.
\end{equation*}
For each vertex $z$ of the elements in $\mathcal{T}_h$, we define the patch
\begin{equation*}
\omega_z:=\bigcup_{z\in T\in\mathcal{T}_h} T.
\end{equation*}

Let us define the postprocessing operator $\Theta_h:\mathbf{Q}\rightarrow \mathbf{Y}_h$. With this operator at hand, and with the previously defined patch $\omega_z$, we fit a piecewise linear function in the average sense, for any $\bv\in\mathbf{Q}$ at the degrees of freedom of element integrations
\begin{equation*}
\displaystyle\Theta_h\bv(z):=\sum_{T\in\omega_z}\frac{\displaystyle\int_T\bv\,dx}{|\omega_z|},
\end{equation*}
where $|\omega_z|$ denotes the measure of the patch.

The operator $\Theta_h$ satisfies the following properties (see \cite[Lemma 3.2, Theorem 3.3]{MR3047040}).
\begin{lemma}[Properties of the postprocessing operator]\label{postprocessing}
The operator $\Theta_h$ defined above satisfies the following:
\begin{enumerate}
\item For $\bu\in \H^{1+s}(\Omega)^n$ with $s$ as in Lemma \ref{lmm:add_eigen} and $T\in\mathcal{T}_h$, there holds $$\|\Theta_h\bu-\bu\|_{0,\O}\leq C h_T^{1+s}\|\bu\|_{1+s,\O},$$
\item $\Theta_h\mathcal{P}_h^0\bv=\Theta_h\bv$,
\item $\|\Theta_h\bv\|_{0,\O}\leq C\|\bv\|_{0,\O}$ for all $\bv\in\mathbf{Q}$,
\end{enumerate}
where the generic constant $C$ is positive and independent of $h$.
\end{lemma}
The following result, proved in  \cite[Theorem 3.3]{ MR3047040} states  a superconvergence property for $\Theta_h$.

\begin{lemma}[superconvergence]
\label{lmm:super}
For $h$ small enough, there exists a positive constant $C$, independent of $h$ and $\lambda$, such that
\begin{equation*}
\|\Theta_h\bu_h-\bu\|_{0,\O}\leq C h^{s}\left(\|\boldsymbol{\rho}-\boldsymbol{\rho}_h\|_{0,\O}+\|\bu-\bu_h\|_{0,\O}\right)+\|\Theta_h\bu-\bu\|_{0,\O}.
\end{equation*}
\end{lemma}
\begin{proof}
From Lemma \ref{lmm:rodolfo11}  and the properties presented in Lemma \ref{postprocessing} it follows that 
\begin{align*}
\|\Theta_h\bu_h-\bu\|_{0,\O}&\leq C \|\Theta_h\bu_h-\Theta_h \mathcal{P}_h^0\bu\|_{0,\O}+\|\Theta_h \mathcal{P}_h^0\bu-\Theta_h \bu\|_{0,\O}+\|\Theta_h\bu-\bu\|_{0,\O} \\
&\leq C \|\bu_h- \mathcal{P}_h^0\bu\|_{0,\O}+\|\Theta_h\bu-\bu\|_{0,\O}\\
&\leq C h^{s}\left(\|\boldsymbol{\rho}-\boldsymbol{\rho}_h\|_{0,\O}+\|\bu-\bu_h\|_{0,\O}\right)+\|\Theta_h\bu-\bu\|_{0,\O}.
\end{align*}
This concludes the proof.
\end{proof}

\section{A posteriori error analysis}
\label{sec:apost}
The following section is dedicated to the design and analysis of  an a posteriori error estimator for our 
mixed eigenvalue problem. 
The main difficulty in the a posteriori error analysis for eigenvalue problems is to control the so called high order terms. To do this task, we adapt the results of \cite{MR3047040} in order to obtain a superconvergence result and hence, prove the desire estimates for our estimator.

\subsection{Properties of the mesh}
For $T\in\mathcal{T}_h$, let $\mathcal{E}(T)$ be the set of its edges/faces, and let $\mathcal{E}_h$ be the set of all
the faces/edges of the triangulation $\mathcal{T}_h$. With these definitions at hand, we write $\mathcal{E}_h=\mathcal{E}_h(\O)\cup\mathcal{E}_h(\partial\O)$, where
\begin{equation*}
\mathcal{E}_h(\O):=\{e\in\mathcal{E}_h\,:\,e\subseteq\O\}\quad\text{and}\quad\mathcal{E}_h(\partial\O):=\{e\in\mathcal{E}_h\,:\, e\subseteq\partial\O\}.
\end{equation*}

On the other hand, for each face/edge $e\in\mathcal{E}_h$ we fix a unit normal vector $\bn_e$ to $e$. Moreover, given $\btau\in\mathbb{H}(\curl, \Omega)$ and $e\in\mathcal{E}_h(\O)$, we let $\jumpp{\btau\times\bn_e}$ be the corresponding jump of the tangential traces across $e$, that is
\begin{equation*}
\jumpp{\btau\times\bn_e}:=(\btau|_T-\btau|_{T'})\big|_e\times\bn_e,
\end{equation*} 
where $T$ and $T'$ are two elements of the triangulation with common edge $e$.
\subsection{Definitions and technical results}
\label{sub:bubbles}
We begin by introducing the bubble functions for two dimensional elements. Given $T\in\mathcal{T}_h$ and $e\in\mathcal{E}(T)$, we let $\psi_T$ and $\psi_e$ be the usual triangle-bubble and edge-bubble functions, respectively (see \cite{MR3059294} for further details about these functions), which satisfy the following properties\\

\begin{enumerate}
\item $\psi_T\in\mathrm{P}_{\ell}(T)$, with $\ell=3$ for 2D or $\ell=4$ for 3D, $\text{supp}(\psi_T)\subset T$, $\psi_T=0$ on $\partial T$ and $0\leq\psi_T\leq 1$ in $T$;
\item $\psi_e|_T\in\mathrm{P}_{\ell}(T)$,  with $\ell=2$ for 2D or $\ell=3$ for 3D, $\text{supp}(\psi_e)\subset \omega_e:=\cup\{T'\in\mathcal{T}_h\,:\, e\in\mathcal{E}(T')\}$, $\psi_e=0$ on $\partial T\setminus e$ and $0\leq\psi_e\leq 1$ in $\omega_e$.
\end{enumerate}

The following properties, proved in \cite[Lemma 1.3]{MR1284252} for an arbitrary polynomial order of approximation, hold.
\begin{lemma}[Bubble function properties]
\label{lmm:bubble_estimates}
Given $\ell\in\mathbb{N}\cup\{0\}$, and for each $T\in\mathcal{T}_h$ and $e\in\mathcal{E}(T)$, there hold
\begin{equation*}
\|\psi_T q\|_{0,T}^2\leq \|q\|_{0,T}^2\leq C \|\psi_T^{1/2} q\|_{0,T}^2\quad\forall q\in\mathrm{P}_{\ell}(T),
\end{equation*}
\begin{equation*}
\|\psi_e L(p)\|_{0,e}^2\leq \| p\|_{0,e}^2\leq C \|\psi_e^{1/2} p\|_{0,e}^2\quad\forall p\in\mathrm{P}_{\ell}(e),
\end{equation*}
and 
\begin{equation*}
h_e\|p \|_{0,e}^2\leq C \|\psi_e^{1/2} L(p)\|_{0,T}^2\leq C
 h_e\|p\|_{0,e}^2\quad\forall p\in\mathrm{P}_{\ell}(e),
 \end{equation*}
 where $L: C(e)\rightarrow C(T)$ with $L(p)\in\mathrm{P}_k(T)$ and $L(p)|_e=p$ for all $p\in\mathrm{P}_k(e)$, and  the hidden constants depend on $k$ and the shape regularity of the triangulation.
 \end{lemma}
Also, we requiere the following technical result (see \cite[Theorem 3.2.6]{MR0520174}).

\begin{lemma}[Inverse inequality]\label{inversein}
Let $l,m\in\mathbb{N}\cup\{0\}$ such that $l\leq m$. Then, for each $T\in\mathcal{T}_h$ there holds
\begin{equation*}
|q|_{m,T}\leq C h_T^{l-m}|q|_{l,T}\quad\forall q\in\mathrm{P}_k(T),
\end{equation*}
where the hidden constant depends on $k,l,m$ and the shape regularity of the triangulations.
\end{lemma}
Let  $I_{h}:\H^{1}(\O)\rightarrow C_I$, where $
	C_I:=\{v\in \mathcal{C}(\bar{\Omega}): v|_{T}\in P_{1}(T) \;\ \forall T\in\CT_{h}\}
	$, be the Cl\'ement interpolant of degree $k=1$. We also define $\boldsymbol{I}_h:[\H^1(\Omega)]^n\rightarrow [C_I]^n=\mathbf{Y}_h $ as the corresponding vectorial version of $I_h$.

The following auxiliary results, available in \cite{MR3453481}, are necessary in our forthcoming analysis.

The following
lemma establishes the local approximation properties of $I_{h}$.
\begin{lemma}
\label{I:clemont}
There exist constants $c_{1}$, $c_{2}>0$, independent of $h$, such that for all $v\in H^{1}(\O)$ there holds
\begin{equation*}
\|v-I_{h}v\|_{0,T}\leq c_1 h_{T}\|v\|_{1,\omega_{T}}\quad \forall T\in\CT_{h},
\end{equation*}
and
\begin{equation*}
\|v-I_{h}v\|_{0,e}\leq c_{2}h_{e}^{1/2}\|v\|_{1,\omega_{e}}\quad \forall e\in \mathcal{E}_h,
\end{equation*}
where $\omega_{T}:=\{T'\in\CT_{h}: T' \text { and } T \text{ share an edge}\}$ and $\omega_{e}:=\{T'\in\CT_{h}: e\in \mathcal{E}_{T'}\}$.
\end{lemma}

The following Helmoltz decomposition holds (see \cite[Lemma 4.3]{MR3453481}).

\begin{lemma}
\label{lm:helm}
For each $\btau\in \mathbb{H}(\bdiv,\O)$ there exist $\boldsymbol{z}\in [\H^{2}(\O)]^{n}$ and $\boldsymbol{\chi}\in \mathbb{H}^1(\Omega)$ such that
\begin{equation*}
\btau=\nabla \boldsymbol{z}+\mathbf{curl}{\boldsymbol{\chi}}\quad \text{ in }\O\quad \text{ and }\quad \|\boldsymbol{z}\|_{2,\O}+\|\boldsymbol{\chi}\|_{1,\O}\leq C\|\btau\|_{\bdiv,\O},
\end{equation*}
where $C$ is a positive  constant independent of all the foregoing variables.
\end{lemma}

\subsection{The local and global error indicators}
 In what follows, let $(\l_h,\boldsymbol{\rho}_h, \bu_h)\in\R\times\mathbb{H}_{h,0}\times \mathbf{Q}_h$ be the solution of \eqref{def:spectral_1h}. Now, for each $T\in\mathcal{T}_h$ we define the local error indicator $\eta_T$ as follows

\begin{multline}
\label{eq:local_eta}
\eta_T^2:=\|\Theta_h\bu_h-\bu_h\|_{0,T}^2+h_T^2\left\|\nabla \bu_h-\frac{1}{\mu}\left\{ \rho_h-\frac{\lambda+\mu}{n\lambda+(n+1)\mu}\tr(\rho_h)\mathbb{I} \right\} \right\|^2_{0,T}\\
+h_T^2\left\|\curl\left(\frac{1}{\mu}\left\{\rho_h- \frac{\lambda+\mu}{n\lambda+(n+1)\mu}\tr(\rho_h)\mathbb{I}  \right\} \right)  \right\|_{0,T}^2+\\
+\sum_{e\in\mathcal{E}(T)\cap\mathcal{E}_h(\O)}h_e\left\|\jump{ \frac{1}{\mu}\left\{\rho_h- \frac{\lambda+\mu}{n\lambda+(n+1)\mu}\tr(\rho_h)\mathbb{I} \right\}\times\bn} \right\|_{0,e}^2\\
+\sum_{e\in\mathcal{E}(T)\cap\mathcal{E}_h(\partial\O)}h_e\left\|\frac{1}{\mu}\left\{\rho_h-\frac{\lambda+\mu}{n\lambda+(n+1)\mu}\tr(\rho_h)\mathbb{I}  \right\}\times\bn \right\|_{0,e}^2,
\end{multline}
and the global estimator is defined by
\begin{equation}
\label{eq:global_est}
\eta:=\left\{ \sum_{T\in\mathcal{T}_h}\eta_T^2\right\}^{1/2}.
\end{equation}

\subsection{Error indicator for the limit problem}
\label{subsec:limit}
It is well known that when $\nu=1/2$, the Lam\'e constant $\lambda$ goes to infinity. This 
behavior leads to a new spectral problem called \emph{the perfectly incompressible elasticity eigenvalue problem}.
In \cite{inzunza2021displacementpseudostress} a complete analysis of this problem is performed.

In the present context of a posteriori error estimates, a slight difference from this reference is needed. To make matters precise, 
the limit problem for the a posteriori analysis is based in the definition of \eqref{eq:a_original}, where we need to compute the limit when $\lambda\rightarrow\infty$. From this computation we have

\begin{equation*}
a_{\infty}(\bxi,\btau)=\lim_{\lambda\rightarrow\infty}a(\bxi,\btau)=\frac{1}{\mu}\int_{\Omega}\bxi:\btau-\frac{1}{n\mu}\int_{\O}\tr(\bxi)\tr(\btau)\quad\forall\bxi,\btau\in\mathbb{H},
\end{equation*}
and hence, problem \eqref{def:elast_system_rho} is rewritten as follows
\begin{equation*}\label{def:elast_system_rho_limit}
\left\{
\begin{array}{rcll}
\displaystyle\frac{1}{\mu}\left\{\boldsymbol{\rho}-\frac{1}{n}\tr(\boldsymbol{\rho})\mathbb{I} \right\}& = &\nabla\bu&  \text{ in } \quad \Omega, \\
\div\boldsymbol{\rho} & = & -\kappa\bu & \text{ in } \quad \Omega, \\
\bu & = & \mathbf{0} & \text{ on } \quad \partial\Omega.
\end{array}
\right.
\end{equation*}

Then, the local error indicator for the limit problem is defined as follows
\begin{multline*}
\eta_{T,\infty}^2:=\|\Theta_h\bu_h-\bu_h\|_{0,T}^2+h_T^2\left\|\nabla \bu_h-\frac{1}{\mu}\left\{ \rho_h-\frac{1}{n}\tr(\rho_h)\mathbb{I} \right\} \right\|^2_{0,T}\\
+h_T^2\left\|\curl\left(\frac{1}{\mu}\left\{\rho_h- \frac{1}{n}\tr(\rho_h)\mathbb{I}  \right\} \right)  \right\|_{0,T}^2+\\
+\sum_{e\in\mathcal{E}(T)\cap\mathcal{E}_h(\O)}h_e\left\|\jump{ \frac{1}{\mu}\left\{\rho_h- \frac{1}{n}\tr(\rho_h)\mathbb{I} \right\}\times\bn} \right\|_{0,e}^2\\
+\sum_{e\in\mathcal{E}(T)\cap\mathcal{E}_h(\partial\O)}h_e\left\|\frac{1}{\mu}\left\{\rho_h-\frac{1}{n}\tr(\rho_h)\mathbb{I}  \right\}\times\bn \right\|_{0,e}^2,
\end{multline*}
and the global estimator is defined by
\begin{equation}
\label{eq:global_est_limit}
\eta_{\infty}:=\left\{ \sum_{T\in\mathcal{T}_h}\eta_{T,\infty}^2\right\}^{1/2}.
\end{equation}

\subsection{Reliability}
In this section we provide an upper bound for the proposed estimator  \eqref{eq:global_est}. We begin with the following estimate for the error

\begin{lemma}
\label{lema_cota_s}
Let $(\kappa,\boldsymbol{\rho}, \bu)\in\R\times\mathbb{H}_{0}\times \mathbf{Q}$ be the solution of \eqref{def:spectral_H0}  and let $(\kappa_{h}$,$\boldsymbol{\rho}_h$,$\bu_h)\in\R\times\mathbb{H}_{0,h}\times\mathbf{Q}_h$ be its finite element approximation, given as the  solution of  \eqref{def:spectral_1h}.  Then for all $\btau\in \mathbb{H}_{0}$ we have.
\begin{multline}
\label{eq:error_bound1}
\|\boldsymbol{\rho}-\boldsymbol{\rho}_{h}\|_{\bdiv,\O}+\|\bu-\bu_{h}\|_{0,\O}\leq C\left(\displaystyle\sup_{\underset{\btau\neq\boldsymbol{0}}{\btau\in\mathbb{H}_{0}}}\frac{-a(\boldsymbol{\rho}_{h},\btau)-b(\btau,\bu_h)}{\|\btau\|_{\bdiv,\O}}\right.\\
\left.+\underbrace{|\kappa_{h}-\kappa |+\|\bu-\Theta_h\bu_{h}\|_{0,\O}}_{\text{h.o.t}}+\|\Theta_h\bu_{h}-\bu_{h}\|_{0,\O}\right).
\end{multline}
\end{lemma}
\begin{proof}
Applying the inf-sup condition \eqref{eq:complete_infsup} on the errors $\boldsymbol{\rho}-\boldsymbol{\rho}_h$ and $\bu-\bu_h$ we have
that there exists a constant $C>0$ such that  
\begin{align*}
\|(\boldsymbol{\rho}\hspace*{-0.045cm}-\hspace*{-0.045cm}\boldsymbol{\rho}_{h},\bu\hspace*{-0.045cm}-\hspace*{-0.045cm}\bu_{h})\|_{\mathbb{H}_{0}\times \mathbf{Q}}\leq &C\left(\hspace*{-0.1cm}\displaystyle\sup_{\underset{(\btau,\bv)\neq \boldsymbol{0}}{(\btau,\bv)\in \mathbb{H}_{0}\times\mathbf{Q}}}\frac{a(\boldsymbol{\rho}\hspace*{-0.045cm}-\hspace*{-0.045cm}\boldsymbol{\rho}_{h},\btau)\hspace*{-0.045cm}+\hspace*{-0.045cm}b(\btau,\bu\hspace*{-0.045cm}-\hspace*{-0.045cm}\bu_h)\hspace*{-0.045cm}+\hspace*{-0.045cm}b(\boldsymbol{\rho}\hspace*{-0.045cm}-\hspace*{-0.045cm}\boldsymbol{\rho}_h,\bv)}{\|(\btau,\bv)\|_{\mathbb{H}_{0}\times \mathbf{Q}}}\right)\\
\leq &C\left(\displaystyle\sup_{\underset{\btau\neq\boldsymbol{0}}{\btau\in\mathbb{H}_{0}}}\frac{-a(\boldsymbol{\rho}_{h},\btau)-b(\btau,\bu_h)}{\|\btau\|_{\bdiv,\O}}+\displaystyle\sup_{\underset{\bv\neq\boldsymbol{0}}{\bv\in\mathbf{Q}}}\frac{b(\boldsymbol{\rho}-\boldsymbol{\rho}_h,\bv)}{\|\bv\|_{0,\O}}\right),
\end{align*}
where we have used the first equation of \eqref{def:spectral_H0}. Now, according to the definition of the bilinear operator $b(\cdot,\cdot)$, the second equation of \eqref{def:spectral_H0} and that $\div(\boldsymbol{\rho}_{h})=-\kappa_h\bu_h$,  and finally using the Cauchy–Schwarz inequality, we get
\begin{align*}
\displaystyle\sup_{\underset{\bv\neq\boldsymbol{0}}{\bv\in\mathbf{Q}}}\frac{b(\boldsymbol{\rho}-\boldsymbol{\rho}_h,\bv)}{\|\bv\|_{0,\O}}&\leq \|\kappa_h\bu_h-\kappa\bu\|_{0,\O}\\
&\leq |\kappa_h-\kappa|\|\bu_h\|_{0,\O}+|\kappa|\|\bu-\bu_h\|_{0,\O}\\
&\leq |\kappa_h-\kappa|\|\bu_h\|_{0,\O}+|\kappa|\left(\|\bu-\Theta_h\bu_h\|_{0,\O}+\|\Theta_h\bu_h-\bu_h\|_{0,\O}\right).
\end{align*}
Then, using the above estimate together that $\|\bu_h\|_{0,\O}=1$ we have 
\begin{align*}
\|\boldsymbol{\rho}-\boldsymbol{\rho}_{h}\|_{\bdiv,\O}+\|\bu-\bu_{h}\|_{0,\O}\leq &C\left(\displaystyle\sup_{\underset{\btau\neq\boldsymbol{0}}{\btau\in\mathbb{H}_{0}}}\frac{-a(\boldsymbol{\rho}_{h},\btau)-b(\btau,\bu_h)}{\|\btau\|_{\bdiv,\O}}\right.\\
&\left.+\underbrace{|\kappa_{h}-\kappa |+\|\bu-\Theta_h\bu_{h}\|_{0,\O}}_{\text{h.o.t}}+\|\Theta_h\bu_{h}-\bu_{h}\|_{0,\O}\right).
\end{align*}
This concludes the proof.
\end{proof}
\begin{remark}\label{eq:hot}
We note that, thanks to Lemmas \ref{lema:apriorie}, \ref{postprocessing} and \ref{lmm:super}, the estimate for the high order term
\begin{equation*}
\text{h.o.t}\leq C h^s \left(\|\boldsymbol{\rho}-\boldsymbol{\rho}_{h}\|_{0,\O}+\|\bu-\bu_{h}\|_{0,\O}\right)
+\|\bu-\Theta_h\bu\|_{0,\O}
\leq Ch^{2s},
\end{equation*}
holds, where the constant $C$ is uniform on $\lambda$ and $h$.
\end{remark}

Our next goal is to bound the supremum in Lemma  \ref{lema_cota_s}. To do this task,
let $\btau\in \mathbb{H}_{0}$ as above  lemma, using the Helmholtz decomposition of $\btau$ given by Lemma \ref{lm:helm}, i.e, $\btau=\nabla \boldsymbol{z}+\curl{\boldsymbol{\chi}}$,  suggests defining $\btau_h\in\mathbb{H}_{h}$ through  a discrete Helmholtz decomposition, as follows
\begin{equation*}
\btau_{h}:=\bPi_h^1\left(\nabla\boldsymbol{z}\right)+\curl(\boldsymbol{\chi}_{h})-d_{h}\mathbb{I},
\end{equation*}
where 
$\boldsymbol{\chi}_{h}:=(\boldsymbol{\chi}_{1h}, \ldots,\boldsymbol{\chi}_{nh})^{t}$, with $\boldsymbol{\chi}_{ih}:=\boldsymbol{I}_{h}(\boldsymbol{\chi}_{i})$ for $i=\{1,...,n\}$, $\bPi_h^1$ is the Raviart-Thomas interpolation operator that satisfies properties  \eqref{daniel1}-\eqref{daniel4}, and the constant $d_{h}$ is chosen in the following way
\begin{align*}
d_{h}:=\dfrac{1}{n|\O|}\int_{\O}\tr(\btau_{h})&=\dfrac{1}{n|\O|}\int_{\O}\tr\left(\bPi_h^1\left(\nabla\boldsymbol{z}\right)+\curl(\boldsymbol{\chi}_{h})\right)\\
&=-\dfrac{1}{n|\O|}\int_{\O}\tr\left(\nabla\boldsymbol{z}-\bPi_h^1\left(\nabla\boldsymbol{z}\right)+\curl(\boldsymbol{\chi}-\boldsymbol{\chi}_{h})\right),
\end{align*}
in order to admit that $\btau_{h}\in \mathbb{H}_{h,0}$. Notice that  we have used the fact that  $\btau\in  \mathbb{H}_{0}$ and its Helmoltz decomposition.

 As a first step to bound the supremum appearing on the right hand side of \eqref{eq:error_bound1}, we  note that for all $\boldsymbol{\xi}_{h}\in\mathbb{H}_{0,h}$, thanks to the first equation of \eqref{def:spectral_1h}, there holds
\begin{equation*}
a(\boldsymbol{\rho}_{h},\boldsymbol{\xi}_{h})+b(\boldsymbol{\xi}_{h},\bu_h)=0.
\end{equation*}
Let $\boldsymbol{\xi}\in\mathbb{H}$ be such that $$\bxi:=\btau-\btau_{h}=\nabla\boldsymbol{z}-\bPi_h^1\left(\nabla\boldsymbol{z}\right)+\curl(\boldsymbol{\chi}-\boldsymbol{\chi}_{h})+d_{h}\mathbb{I}.$$
Since $\textbf{div}(\textbf{curl}(\boldsymbol{\chi}-\boldsymbol{\chi}_{h}))=\textbf{div}(d_{h}\mathbb{I})=0$, then $\textbf{div}(\bxi)=\textbf{div}(\nabla\boldsymbol{z}-\bPi_h^1\left(\nabla\boldsymbol{z}\right))=\textbf{div}(\nabla\boldsymbol{z})-\mathcal{P}_h^k(\bdiv(\nabla\boldsymbol{z}))$ (see \eqref{eq:commutative}) and using that $\mathcal{P}_h^k$ is the $\L^2(\O)$-orthogonal projector,  we have that
$
b(\bxi,\bu_h)=0.
$
Therefore, from the fact that $\boldsymbol{\rho}_{h}\in\mathbb{H}_{0,h}$ we  obtain the following identity
\begin{align*}
-\left[a(\boldsymbol{\rho}_{h},\btau)+b(\btau,\bu_h)\right]&=-\left[a(\boldsymbol{\rho}_{h},\bxi)+b(\bxi,\bu_h)\right]=-a(\boldsymbol{\rho}_{h},\bxi).
\end{align*}
Now, invoking the definition of $\bxi$ and that $a(\boldsymbol{\rho}_{h},d_{h}\mathbb{I})=d_{h}a(\boldsymbol{\rho}_{h},\mathbb{I})=0$ we obtain  
\begin{multline}\label{eq:sup}
-\left[a(\boldsymbol{\rho}_{h},\btau)+b(\btau,\bu_h)\right]\\
=\underbrace{-a(\boldsymbol{\rho}_{h},\nabla\boldsymbol{z}-\bPi_h^1\left(\nabla\boldsymbol{z}\right))}_{\mathfrak{T}_1}
+\underbrace{-a(\boldsymbol{\rho}_{h},\curl(\boldsymbol{\chi}-\boldsymbol{\chi}_{h}))}_{\mathfrak{T}_{2}}.\end{multline}

The following step is to bound the terms $\mathfrak{T}_{1}$ and $\mathfrak{T}_2$. We begin with $\mathfrak{T}_1$.
\begin{lemma}
\label{lm_a1}
There exists $C>0$, independent of $\lambda$ and $h$, such that
\begin{equation*}
\left| \mathfrak{T}_1\right|\leq C\left\{\sum_{T\in\CT_{h}}\eta_{T}^{2} \right\}^{1/2}\|\btau\|_{\bdiv,\O}.
\end{equation*}
\end{lemma}
\begin{proof}
From the definition of   $\boldsymbol{\rho}_{h}^{\texttt{d}}$ and the identity  $\tr(\btau)=\btau:\mathbb{I}$ we obtain 
\begin{align*}
\mathfrak{T}_1=-\int_{\Omega}\frac{1}{\mu}\left\{\boldsymbol{\rho}_{h}-\dfrac{\lambda+\mu}{n\lambda+(n+1)\mu}\tr(\boldsymbol{\rho}_{h})\mathbb{I}\right\}:(\nabla\boldsymbol{z}-\bPi_h^1\left(\nabla\boldsymbol{z}\right)).
\end{align*}
On the other hand, since $\bu_h\in \mathbf{P}_0(T)$, for all $T\in\CT_h$, we obtain 
\begin{align*}
\int_{\Omega}\nabla\bu_h:(\nabla\boldsymbol{z}-\bPi_h^1\left(\nabla\boldsymbol{z}\right))=0.
\end{align*}
Therefore, using \eqref{daniel2} and Lemma \ref{lm:helm}, we have
\begin{align*}
\left| \mathfrak{T}_1\right|&\leq \sum_{T\in \CT_{h}}h_{T}\left\|\nabla\bu_h-\frac{1}{\mu}\left\{\boldsymbol{\rho}_{h}-\dfrac{\lambda+\mu}{n\lambda+(n+1)\mu}\tr(\boldsymbol{\rho}_{h})\mathbb{I}\right\}\right\|_{0,T}\|\nabla\boldsymbol{z}\|_{1,T}\\
&\leq C\left\{\sum_{T\in\CT_{h}}h_{T}^{2}\left\|\nabla\bu_h-\frac{1}{\mu}\left\{\boldsymbol{\rho}_{h}-\dfrac{\lambda+\mu}{n\lambda+(n+1)\mu}\tr(\boldsymbol{\rho}_{h})\mathbb{I}\right\}\right\|_{0,T}^{2} \right\}^{1/2}\|\btau\|_{\bdiv,\O}\\
&\leq C\left\{\sum_{T\in\CT_{h}}\eta_{T}^{2} \right\}^{1/2}\|\btau\|_{\bdiv,\O}.
\end{align*}
This concludes the proof.
\end{proof}

The bound for  $\mathfrak{T}_2$ contained in the following lemma.
\begin{lemma}
\label{lm_a2}
There exists $C>0$, independent of $\lambda$ and $h$, such that
\begin{equation*}
\left|\mathfrak{T}_2\right|\leq C\left\{\sum_{T\in\CT_{h}}\eta_{T}^{2} \right\}^{1/2}\|\btau\|_{\bdiv,\O}.
\end{equation*}
\end{lemma}
\begin{proof}
Integrating by parts on each $T\in\CT_{h}$, we obtain that
\begin{align*}
\mathfrak{T}_2
&=-\sum_{T\in \CT_{h}}\left[\int_{T}\curl\left(\dfrac{1}{\mu}\left\{\boldsymbol{\rho}_{h}-\dfrac{\lambda+\mu}{n\lambda+(n+1)\mu}\tr(\boldsymbol{\rho}_{h})\mathbb{I}\right\}\right):\left(\boldsymbol{\chi}-\boldsymbol{\chi}_{h}\right)\right.\\
&+\left.\int_{\partial E}\left(\dfrac{1}{\mu}\left\{\boldsymbol{\rho}_{h}-\dfrac{\lambda+\mu}{n\lambda+(n+1)\mu}\tr(\boldsymbol{\rho}_{h})\mathbb{I}\right\}\right)\times\bn:\left(\boldsymbol{\chi}-\boldsymbol{\chi}_{h}\right)\right]\\
&=-\sum_{T\in \CT_{h}}\int_{T}\curl\left(\dfrac{1}{\mu}\left\{\boldsymbol{\rho}_{h}-\dfrac{\lambda+\mu}{n\lambda+(n+1)\mu}\tr(\boldsymbol{\rho}_{h})\mathbb{I}\right\}\right):\left(\boldsymbol{\chi}-\boldsymbol{\chi}_{h}\right)\\
&-\sum_{e\in \mathcal{E}_{h}(\O)}\int_{e}\jump{ \frac{1}{\mu}\left\{\boldsymbol{\rho}_h- \dfrac{\lambda+\mu}{n\lambda+(n+1)\mu}\tr(\rho_h)\mathbb{I} \right\}\times\bn}:\left(\boldsymbol{\chi}-\boldsymbol{\chi}_{h}\right)\\
&-\sum_{e\in \mathcal{E}_{h}(\partial \O)}\int_{e}\left( \frac{1}{\mu}\left\{\boldsymbol{\rho}_h- \dfrac{\lambda+\mu}{n\lambda+(n+1)\mu}\tr(\rho_h)\mathbb{I} \right\}\times\bn\right):\left(\boldsymbol{\chi}-\boldsymbol{\chi}_{h}\right).
\end{align*}
Applying Cauchy-Schwarz inequality, recalling that $\boldsymbol{\chi}_{h}=\boldsymbol{I}_{h}\boldsymbol{\chi}$, and  invoking the approximation properties presented in  Lemma \ref{I:clemont} and Lemma \ref{lm:helm},  we have
\begin{align*}
|\mathfrak{T}_2|
&\leq \sum_{T\in \CT_{h}}h_{T}\left\|\curl\left(\dfrac{1}{\mu}\left\{\boldsymbol{\rho}_{h}-\dfrac{\lambda+\mu}{n\lambda+(n+1)\mu}\tr(\boldsymbol{\rho}_{h})\mathbb{I}\right\}\right)\right\|_{0,T}\|\boldsymbol{\chi}\|_{1,\omega_{T}}\\
&+\sum_{e\in \mathcal{E}(T)\cap\mathcal{E}_{h}(\O)}h_{e}\left\|\jump{ \frac{1}{\mu}\left\{\boldsymbol{\rho}_h- \dfrac{\lambda+\mu}{n\lambda+(n+1)\mu}\tr(\rho_h)\mathbb{I} \right\}\times\bn}\right\|_{0,e}\|\boldsymbol{\chi}\|_{1,\omega_{e}}\\
&+\sum_{e\in \mathcal{E}(T)\cap\mathcal{E}_{h}(\partial \O)}h_{e}\left\| \frac{1}{\mu}\left\{\boldsymbol{\rho}_h- \dfrac{\lambda+\mu}{n\lambda+(n+1)\mu}\tr(\rho_h)\mathbb{I} \right\}\times\bn\right\|_{0,e}\|\boldsymbol{\chi}\|_{1,\omega_{e}}\\
&\leq C\left\{\sum_{T\in\CT_{h}}\eta_{T}^{2} \right\}^{1/2}\|\btau\|_{\bdiv,\O}.
\end{align*}
This concludes the proof.
\end{proof}
As a consequence of Lemma \ref{lema:apriorie}, Lemma \ref{lema_cota_s}, Remark \ref{eq:hot}, estimate \eqref{eq:sup}, Lemmas \ref{lm_a1} and \ref{lm_a2}, and the definition of the local estimator $\eta_T$, we have the following result 
\begin{prop}
Let $(\kappa,\boldsymbol{\rho}, \bu)\in\R\times\mathbb{H}_{0}\times \mathbf{Q}$ be the solution of \eqref{def:spectral_H0}  and let $(\kappa_{h},\boldsymbol{\rho}_h,\bu_h)\in\R\times\mathbb{H}_{0,h}\times\mathbf{Q}_h$ solution of  \eqref{def:spectral_1h}. Then,  there exist positive constants $C$ and $h_0$, with $C$ independent of $h$ and $\lambda$, such that, for all $h < h_0$, there holds.
\begin{align*}
\|\boldsymbol{\rho}-\boldsymbol{\rho}_{h}\|_{\bdiv,\O}+\|\bu-\bu_{h}\|_{0,\O}&\leq C\left(\left\{\sum_{T\in\CT_{h}}\eta_{T}^{2} \right\}^{1/2}+\|\bu-\Theta_h\bu\|_{0,\O}\right),\\
|\kappa_{h}-\kappa |&\leq C\left(\sum_{T\in\CT_{h}}\eta_{T}^{2}+\|\bu-\Theta_h\bu\|_{0,\O}^2\right).
\end{align*}
\end{prop}

\subsection{Efficiency}
The aim of this section is to obtain a lower bound for the local indicator \eqref{eq:local_eta}. To do this task, we will apply
the localization technique based in bubble functions, together with inverse inequalities. In order to present the material, the efficiency 
will be proved in several steps, where each one of these correspond to one of the terms of \eqref{eq:local_eta}.

We begin by invoking the following result, proved in \cite[Lemma 4.3]{MR2293249} and \cite[Lemma 4.9]{MR3453481} for the two and three dimensional cases, respectively.
\begin{lemma}\label{lmm:invcurl}
Let $\btau_h\in\mathbb{L}^2(\O)$ be a piecewise polynomial of degree $k\geq 0$ on each $T\in\mathcal{T}_h$ such that approximates $\btau\in\mathbb{L}^2(\O)$, where  $\mathbf{curl}(\btau)=\boldsymbol{0}$ on each $T\in\mathcal{T}_h$. Then, there exists  a positive constant $C$, independent of $h$ and $\lambda$, such that 
\begin{equation*}
\|\mathbf{curl}(\btau_h)\|_{0,T}\leq C h_T^{-1}\|\btau-\btau_h\|_{0,T}\quad \forall T\in\mathcal{T}_h.
\end{equation*}
\end{lemma}

Now our task is to bound each of the contributions of $\eta_T$ in \eqref{eq:local_eta}. We begin with   the term 
$$h_T^2\left\|\nabla\bu_h-\frac{1}{\mu}\left\{ \boldsymbol{\rho}_h-\frac{\lambda+\mu}{n\lambda+(n+1)\mu}\tr(\boldsymbol{\rho}_h)\mathbb{I} \right\} \right\|^2_{0,T}.$$
Given  an element $T\in\CT_h$, let us define $\Upsilon_T:=\nabla\bu_h-\boldsymbol{\chi}_h$ where
\begin{equation*}
\boldsymbol{\chi}_h:=\frac{1}{\mu}\left\{ \boldsymbol{\rho}_h-\frac{\lambda+\mu}{n\lambda+(n+1)\mu}\tr(\boldsymbol{\rho}_h)\mathbb{I} \right\}.
\end{equation*}
Setting 
$$\boldsymbol{\chi}:=\frac{1}{\mu}\left\{ \boldsymbol{\rho}-\frac{\lambda+\mu}{n\lambda+(n+1)\mu}\tr(\boldsymbol{\rho})\mathbb{I} \right\},$$
and using  the relations $\Vert \tr(\boldsymbol{\rho})\Vert_{0,T}\leq \sqrt{n}\Vert\boldsymbol{\rho}\Vert_{0,T}$ and $\frac{\lambda+\mu}{n\lambda+(n+1)\mu}<\frac{1}{n}$, we obtain
$$
\begin{aligned}
\Vert \boldsymbol{\chi}-\boldsymbol{\chi}_h\Vert_{0,T}&\leq\frac{1}{\mu}\left\{\Vert \boldsymbol{\rho}- \boldsymbol{\rho}_{h}\Vert_{0,T}+ \frac{\lambda+\mu}{n\lambda+(n+1)\mu}\Vert\tr(\boldsymbol{\rho}-\boldsymbol{\rho}_h)\Vert_{0,T}\right\}\\
&\leq \frac{1}{\mu}\left\{\Vert \boldsymbol{\rho}- \boldsymbol{\rho}_{h}\Vert_{0,T}+ \frac{\sqrt{n}}{n}\Vert\boldsymbol{\rho}-\boldsymbol{\rho}_h\Vert_{0,T}\right\}\\
&=\frac{n+\sqrt{n}}{n\mu}\Vert \boldsymbol{\rho}- \boldsymbol{\rho}_{h}\Vert_{0,T}.
\end{aligned}
$$
%
%


Since $\nabla\bu=\boldsymbol{\chi}$ and invoking the bubble function $\psi_T$  defined in subsection \ref{sub:bubbles} we have
\begin{align*}
\|\Upsilon_T\|_{0,T}^2&\leq C\|\psi_T^{1/2}\Upsilon_T\|_{0,T}^2=C\int_T\psi_T\Upsilon_T:(\nabla(\bu_h-\bu)+(\boldsymbol{\chi}-\boldsymbol{\chi}_h))\\
&=C\left\{\int_T\bdiv(\psi_T\Upsilon_T)\cdot(\bu-\bu_h)+\int_T\psi_T\Upsilon_T:(\boldsymbol{\chi}-\boldsymbol{\chi}_h) \right\}\\
&\leq C\|\bdiv(\psi_T\Upsilon_T)\|_{0,T}\|\bu-\bu_h\|_{0,T}+\|\psi_T\Upsilon_T\|_{0,T}\|\boldsymbol{\chi}-\boldsymbol{\chi}_h\|_{0,T}\\
&\leq C\big\{ h_T^{-1}\|\bu-\bu_h\|_{0,T}+\|\boldsymbol{\chi}-\boldsymbol{\chi}_h\|_{0,T} \big\}\|\Upsilon_T\|_{0,T}\\
&\leq C\left\{ h_T^{-1}\|\bu-\bu_h\|_{0,T}+\frac{n+\sqrt{n}}{n\mu}\|\boldsymbol{\rho}-\boldsymbol{\rho}_h\|_{0,T} \right\}\|\Upsilon_T\|_{0,T},
\end{align*}
where we have used integration by parts,  Cauchy-Schwarz inequality, Lemmas \ref{lmm:bubble_estimates} and \ref{inversein}, and the properties of $\psi_T$ given by Lemma \ref{lmm:bubble_estimates}. Hence
\begin{equation}\label{eq:effi_1}
h_T^2\left\|\nabla\bu_h\hspace{-0.04cm}-\hspace{-0.04cm}\frac{1}{\mu}\left\{ \boldsymbol{\rho}_h\hspace{-0.04cm}-\hspace{-0.04cm}\frac{\lambda+\mu}{n\lambda+(n+1)\mu}\tr(\boldsymbol{\rho}_h)\mathbb{I} \right\} \right\|^2_{0,T}\leq\hspace{-0.04cm} C\left\{\|\bu\hspace{-0.04cm}-\hspace{-0.04cm}\bu_h\|_{0,T}^2\hspace{-0.04cm}+\hspace{-0.04cm}\|\boldsymbol{\rho}\hspace{-0.04cm}-\hspace{-0.04cm}\boldsymbol{\rho}_h\|_{0,T}^2 \right\},
\end{equation}
where the constant $C$ is independent of $h$ and $\lambda$.

Now, following the proof of \cite[Lemma 4.11]{MR3453481} we can prove that 

\begin{equation*}\label{eq:effi_2}
h_T^2\left\|\curl\left(\frac{1}{\mu}\left\{\rho_h- \frac{\lambda+\mu}{n\lambda+(n+1)\mu}\tr(\rho_h)\mathbb{I}  \right\} \right)  \right\|_{0,T}^2\leq \hat{C}\|\boldsymbol{\rho}-\boldsymbol{\rho}_h\|_{0,T}^2\\
\end{equation*}
and
\begin{equation*}
\label{eq:effi_3}
h_e\left\|\jump{ \frac{1}{\mu}\left\{\boldsymbol{\rho}_h- \frac{\lambda+\mu}{n\lambda+(n+1)\mu}\tr(\boldsymbol{\rho}_h)\mathbb{I} \right\}\times\bn}\right\|_{0,e}^2\leq C\|\boldsymbol{\rho}-\boldsymbol{\rho}_h\|_{0,\omega_e}^2,
\end{equation*}
for all  $e\in \mathcal{E}_{h}(\O)$, and the constants $\hat{C}$ and $C$ are independent of $h$ and $\lambda$.

The following step is to bound the boundary term of the estimator  $\eta$.
	Given $e\in\mathcal{E}_{h}(\partial \O)$, let us define $\Upsilon_e:=\boldsymbol{\chi}_h\times\bn$. Then, applying Lemma \ref{lmm:bubble_estimates} and the extension operator $L$, we obtain
	\begin{align*}
		\|\Upsilon_e\|_{0,e}^2&\leq C\|\psi_e^{1/2}\Upsilon_e\|_{0,e}^2=C\int_e\psi_e\Upsilon_e:\boldsymbol{\chi}_h\times\bn\\
		&=C\left(\int_{\partial T_e}\psi_eL(\Upsilon_e):(\boldsymbol{\chi}_h-\boldsymbol{\chi})\times\bn+\int_e\psi_e\Upsilon_e:\boldsymbol{\chi}\times\bn\right).
	\end{align*}
	Since $\bu=\boldsymbol{0}$ on $\partial\Omega$, then $\nabla u_i$ is parallel to $\bn$ on $e$. Hence, using that $\nabla \bu=\boldsymbol{\chi}$, we have $\boldsymbol{\chi}\times \bn=\boldsymbol{0}$. This fact, together with integration by parts, allow to obtain
		\begin{align*}
			\|\Upsilon_e\|_{0,e}^2&\leq C\int_{\partial T_e}\psi_eL(\Upsilon_e):(\boldsymbol{\chi}_h-\boldsymbol{\chi})\times\bn.\\
			&= C\left(\int_{T_e}\psi_eL(\Upsilon_e):\curl\left(\boldsymbol{\chi}_h\right)+\int_{T_e}\left(\boldsymbol{\chi}-\boldsymbol{\chi}_h\right):\curl\left(\psi_eL(\Upsilon_e)\right)\right)\\
			&\leq C\big(\|\psi_eL(\Upsilon_e)\|_{0,T_e}\|\curl\left(\boldsymbol{\chi}_h\right)\|_{0,T_e}\hspace*{-0.05cm}+\hspace*{-0.05cm}\|\boldsymbol{\chi}-\boldsymbol{\chi}_h\|_{0,T_e}\|\curl\left(\psi_eL(\Upsilon_e)\right)\|_{0,T_e}\big)\\
			&\leq C h_e^{-1/2}\|\boldsymbol{\chi}-\boldsymbol{\chi}_h\|_{0,T_e}\|\Upsilon_e\|_{0,e},
	\end{align*}
	where we have used Lemma \ref{lmm:bubble_estimates}, Lemma \ref{inversein}, and Lemma \ref{lmm:invcurl}. Thus, we have  proved the estimate
	\begin{equation}\label{eq:effi_3}
		h_e\left\| \frac{1}{\mu}\left\{\boldsymbol{\rho}_h- \frac{\lambda+\mu}{n\lambda+(n+1)\mu}\tr(\boldsymbol{\rho}_h)\mathbb{I} \right\}\times\bn\right\|_{0,e}^2\leq C\|\boldsymbol{\rho}-\boldsymbol{\rho}_h\|_{0,T_e}^2\quad \forall e\in \mathcal{E}_{h}(\partial \O).
	\end{equation}

Finally, for the  term $\|\Theta_h\bu_h-\bu_h\|_{0,T}^2$, we add and subtract $\Theta_h\bu$ and $\bu$, apply triangle inequality, and Lemma \ref{postprocessing}, leading to 	
\begin{equation}
	\label{eq:super22}
	\|\Theta_h\bu_h-\bu_h\|_{0,T}^2\leq C\left(\|\bu-\bu_h\|_{0,T}^2+\|\Theta_h\bu_h-\Theta_h\bu\|_{0,T}^2+\|\Theta_h\bu-\bu\|_{0,T}^2\right).
\end{equation}
Note that the last term of \eqref{eq:super22} is asymptotically negligible thanks to Lemma \ref{postprocessing}.

Gathering the previous results, namely \eqref{eq:effi_1}--\eqref{eq:super22}, we are in a position to establish the efficiency $\eta$, which is stated in the following result.
\begin{theorem}[Efficiency]
There exists a constant $C>0$, independent of $h$ and $\lambda$ such that
$$
\eta^2:=\sum_{T\in\CT_{h}}\eta_{T}^{2} \leq C\left( \|\bu-\bu_h\|_{0,\Omega}^2 + \|\boldsymbol{\rho}-\boldsymbol{\rho}_h\|_{0,\Omega}^2+\text{h.o.t}\right).
$$
\end{theorem}
\begin{proof}
The proof is is a consequence of \eqref{eq:effi_1}--\eqref{eq:super22} and Lemma \ref{postprocessing}.
\end{proof}
\begin{remark}
	Notice that all our analysis has been performed considering the estimator  $\eta$ defined in \eqref{eq:global_est}. However, these computations are straightforward when the limit 
	estimator $\eta_{\infty}$ is considered.
\end{remark}


\section{Numerical experiments}
\label{sec:numerics}

In this section we report numerical tests in order to assess the performance of the devised estimators 
$\eta$ and $\eta_{\infty}$ defined in \eqref{eq:global_est} and \eqref{eq:global_est_limit}, respectively.
All the reported results  have been obtained with a FEniCS code \cite{MR3618064}, considering the meshes that this software provides.
We recall that the Lamé coefficients for the elasticity equations are defined by
	$$
	\lambda:=\frac{E\nu}{(1+\nu)(1-2\nu)}\quad\text{and}\quad\mu:=\frac{E}{2(1+\nu)},
	$$
	where $\nu$ is the Poisson ratio and $E$ is the Young's modulus.

In the elasticity test, the Poisson ratio $\nu$ is allowed to take different values.  To make matters precise, if $\nu\rightarrow 1/2$, the Lam\'e constant $\lambda\rightarrow\infty$
	and hence, the definition of $a(\cdot,\cdot)$ changes to $a_{\infty}(\cdot,\cdot)$ as we have claimed in subsection \ref{subsec:limit}. For simplicity,  we will denote the indicators simply by $\eta$ for both cases, $\nu\neq 1/2$ and $\nu_\infty=1/2$. Additionally, the experiments consider a Young's modulus $E = 1$, the boundary condition $\bu=0$, and the lowest order polynomial degree $k=0$.

Throughout this section, we denote by $N$ the number of degrees of freedom, i.e., $N:=\dim(\mathbb{H}\times\mathbf{Q})$. We also set $\omega:=\sqrt{\kappa}$ as the eigenfrequency and $\err(\omega_i)$ denotes the error on the $i$-th eigenfrequency with 
$$
\err(\omega_i):=\vert \omega_{h_i}-\omega_{i}\vert,
$$
whereas the effectivity indexes with respect to $\eta$ and the eigenfrequency $\omega_i$ is defined by
$$
\eff(\omega_i):=\frac{\err(\omega_i)}{\eta^2}.
$$
Here, an exact solution $\omega_i$ will be defined as those accurate values of frequencies that are calculated by extrapolation through the least squares fit of the model
$$
\omega_{hi}\approx \omega_{i} + C_ih^{\alpha_i}.
$$

In order to apply the adaptive finite element method, we shall generate a sequence of nested conforming triangulations using the loop
\begin{center}
	\Large \textrm{solve $\rightarrow$ estimate $\rightarrow$ mark $\rightarrow$ refine,} 
\end{center}
based on \cite{verfuhrt1996}:
\begin{enumerate}
	\item Set an initial mesh $\CT_{h}$.
	\item Solve \eqref{def:spectral_1} in the actual mesh to obtain $\omega_{h}$ and $(\omega_{h},\rho_{h},\bu_{h})$. 	
	\item Compute $\eta_T$ for each $T\in\CT_{h}$ using the eigenfunctions $(\rho_{h},\bu_{h})$. 
	\item Use the maximal marking strategy to refine each $T'\in \CT_{h}$ whose indicator $\eta_{T'}$ satisfies
	$$
	\eta_{T'}\geq \beta\max\{\eta_{T}\,:\,T\in\CT_{h} \},
	$$
	for some $\beta\in(0,1)$.
	\item Set $\CT_{h}$ as the actual mesh and go to step 2.
\end{enumerate}

The refinement algorithm is the one implemented by Fenics through the command \texttt{refine}, which implements Plaza and Carey's algorithms for 2D and 3D geometries. The algorithms use local refinement of simplicial grids based on the skeleton.
			\subsection{Test 1: 3D L-shaped domain} This experiment considers a non-convex polygonal domain with a singularity along an axis. We set
		$$
		\Omega:=(-1,1)\times(-1,1)\times(-1,0)\backslash\bigg((-1,0)\times(-1,0)\times(-1,0) \bigg),
		$$
		which corresponds to a 3D L-shaped domain. The initial mesh of the domain, depicted in Figure \ref{fig:l-3d}, is considered for both, uniform and adaptive refinements.
		\begin{figure}[h]
			\centering
			\includegraphics[scale=0.11,trim=22cm 4.2cm 21cm 2.4cm,clip]{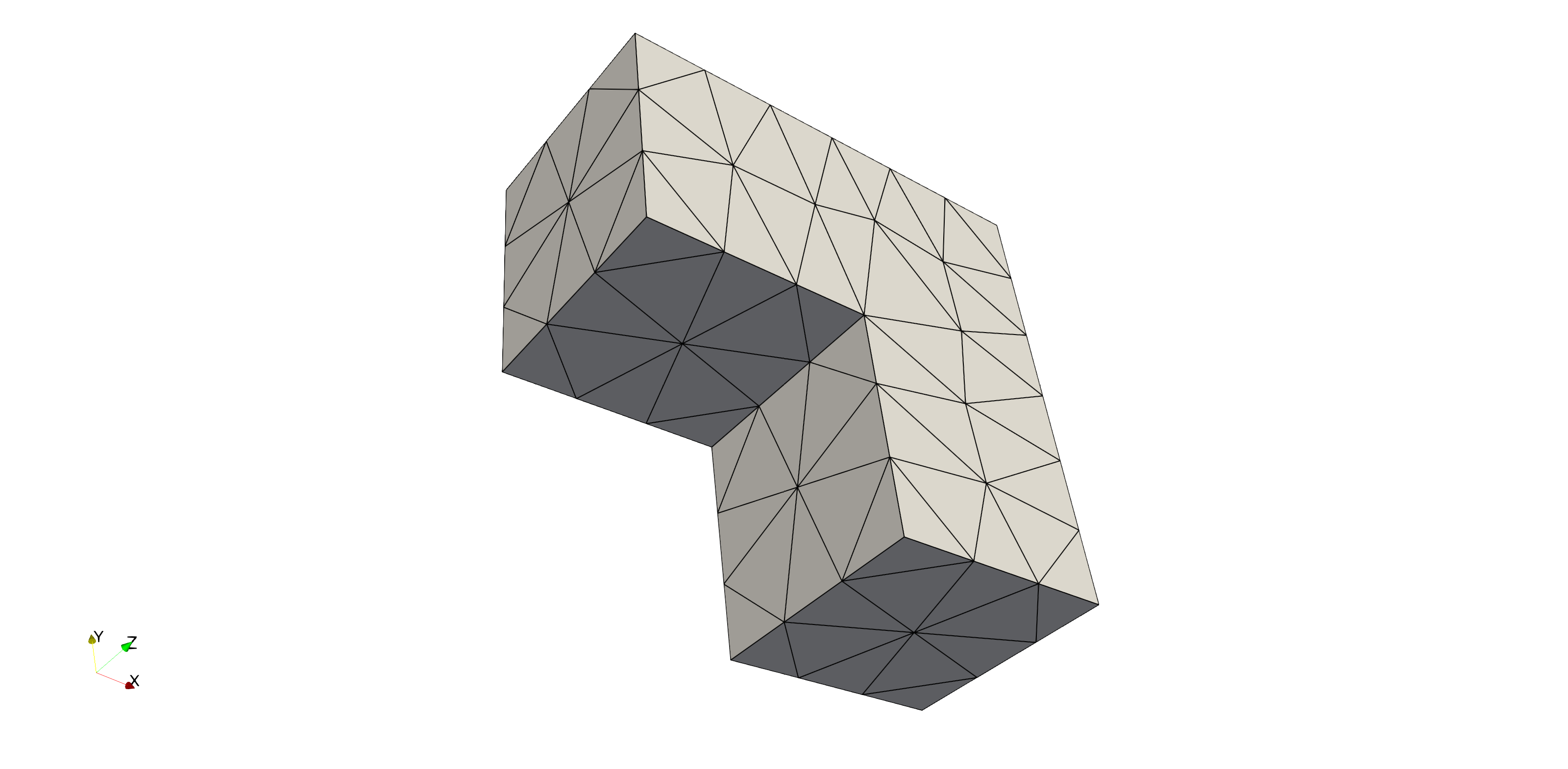}
			\caption{Test 1. The initial three dimensional L-shaped mesh.}
			\label{fig:l-3d}
		\end{figure}
		Since there are no exact eigenvalues for this geometry, we proceed to find an accurate value by means of sufficiently refined uniform meshes and a least-squares fitting. The extrapolated eigenvalues, corresponding to different values of $\nu$ are presented in Table \ref{table: l-3d-tabla-de-extrapolados}. 
		\begin{table}[H]
			\begin{center}
				\begin{tabular}{c|c}
					$\nu$ & $\omega_1$\\\midrule
					0.35 & 3.01757\\
					0.49 & 3.73062 \\
					0.5 & 3.73364 \\
				\end{tabular}
			\end{center}
			\caption{Test 1. Lowest computed eigenvalues using highly refined meshes and least square fitting in the three dimensional L-shaped domain.}
			\label{table: l-3d-tabla-de-extrapolados}
		\end{table}
		In Figure \ref{fig:lshape-3d-error} we report the error curves obtained using uniform and adaptive meshes. The fit line in the uniform refinements shows that the selected values are appropriate to be considered as "accurate" in the calculations performed. The slope of the fitted line on the uniform refined meshes is $-0.40$ ($\nu=0.35$ and $\nu=0.49$), while the slope for $\nu=0.5$ is $-0.42$. This indicates that the errors of the eigenfrequencies computed with uniform meshes satisfy $\vert \omega_1-\omega_{1h}\vert\approx CN^{-0.40}=Ch^{2s}$, with $s=0.6$ ($\nu=0.35$ and $\nu=0.49$), and $\vert \omega_1-\omega_{1h}\vert\approx CN^{-0.42}=Ch^{2s}$, with $s=0.63$ ($\nu=0.5$). Figure \ref{fig:lshape-3d-error} (top) also shows that the eigenfrequencies computed with the adaptive refinement converge to the "exact" eigenfrequency with a higher order of convergence than those obtained with uniform refinement. In this case, the slope of the fitted lines obtained are $-0.65$ ($\nu=0.35$) and $-0.64$ ($\nu=0.49$ and $\nu=0.5$). This implies that the errors bounds behaves like $\mathcal{O}(N^{-0.65})\simeq\mathcal{O}(h^{2s})$, with $s\approx0.98$ ($\nu=0.35$), $\mathcal{O}(N^{-0.64})\simeq\mathcal{O}(h^{2s})$, with $s\approx0.96$ ($\nu=0.49$ and $\nu=0.5$), which shows that the estimator is able to recover the optimal order of convergence for this singular eigenfrequency. Moreover, Figure \ref{fig:lshape-3d-error} (bottom left and bottom right) shows that the square of the estimator behaves like $\mathcal{O}(N^{-2/3})$, hence the effectivity $\mathtt{eff}(\omega_1)$ remains bounded above and below away from zero. 
		
		Note that the singularity is along $(0,y,0)$, and the proposed estimator is able to detect it and refine near this zone. This is shown in Figures \ref{fig:l-3d-nu035-mallas} and \ref{fig:l-3d-nu050-mallas}, where we present different intermediate meshes in the adaptive refinement process for the selected values of $\nu$. 
		
		\begin{figure}
			\centering
			\begin{minipage}{\linewidth}\centering
				\includegraphics[scale=0.4]{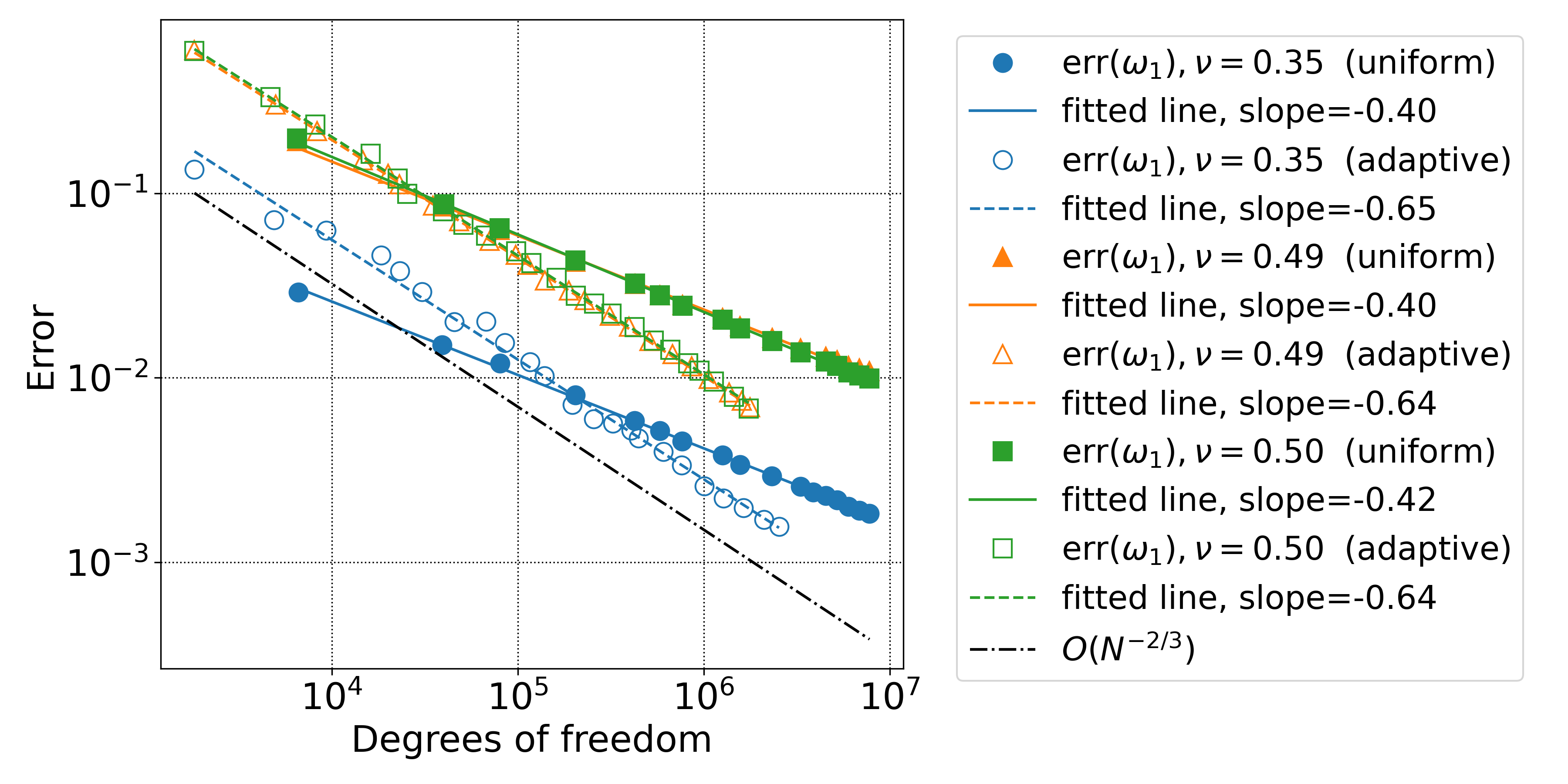}
			\end{minipage}\\
			\begin{minipage}{0.49\linewidth}\centering
				\includegraphics[scale=0.35]{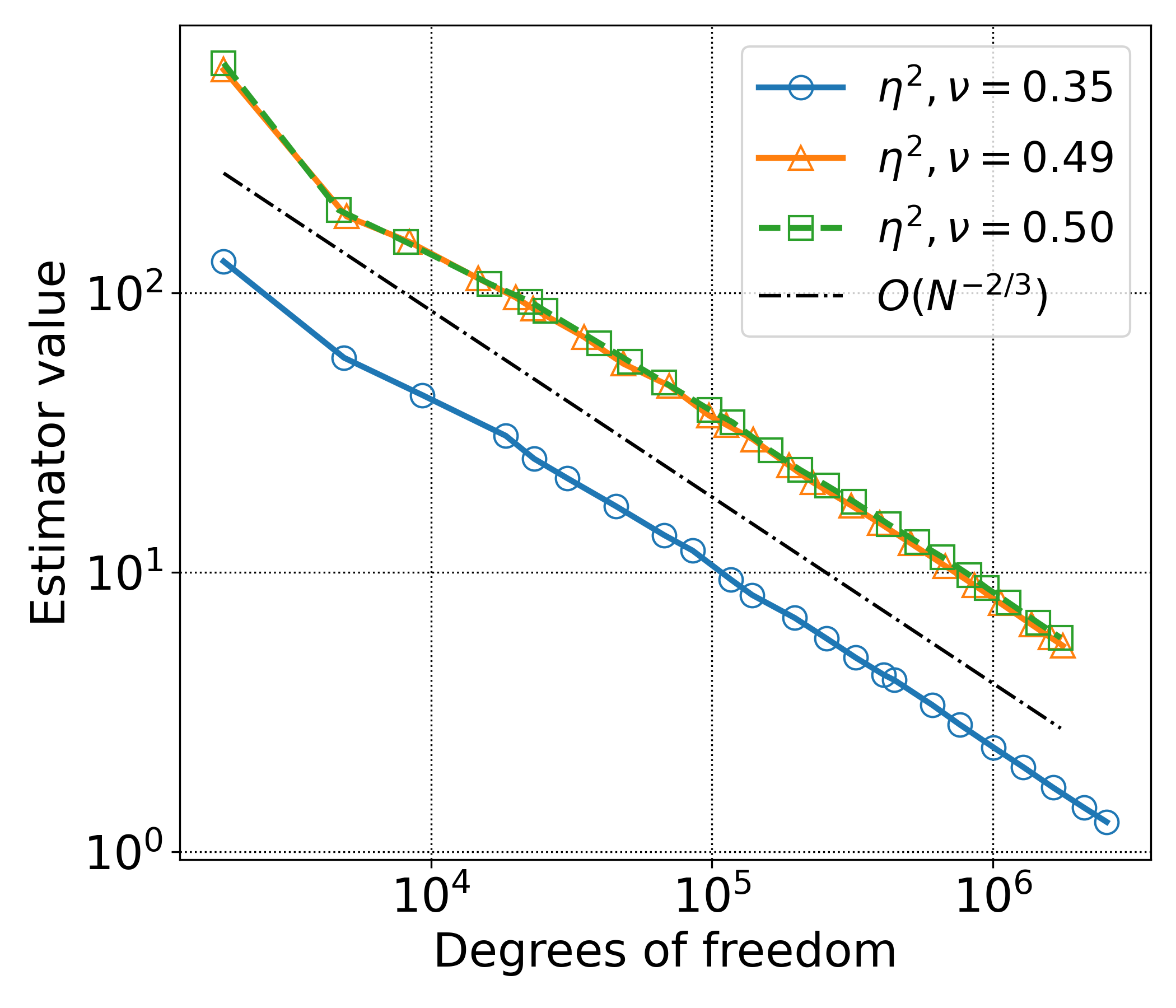}
			\end{minipage}
			\begin{minipage}{0.49\linewidth}\centering
				\includegraphics[scale=0.35]{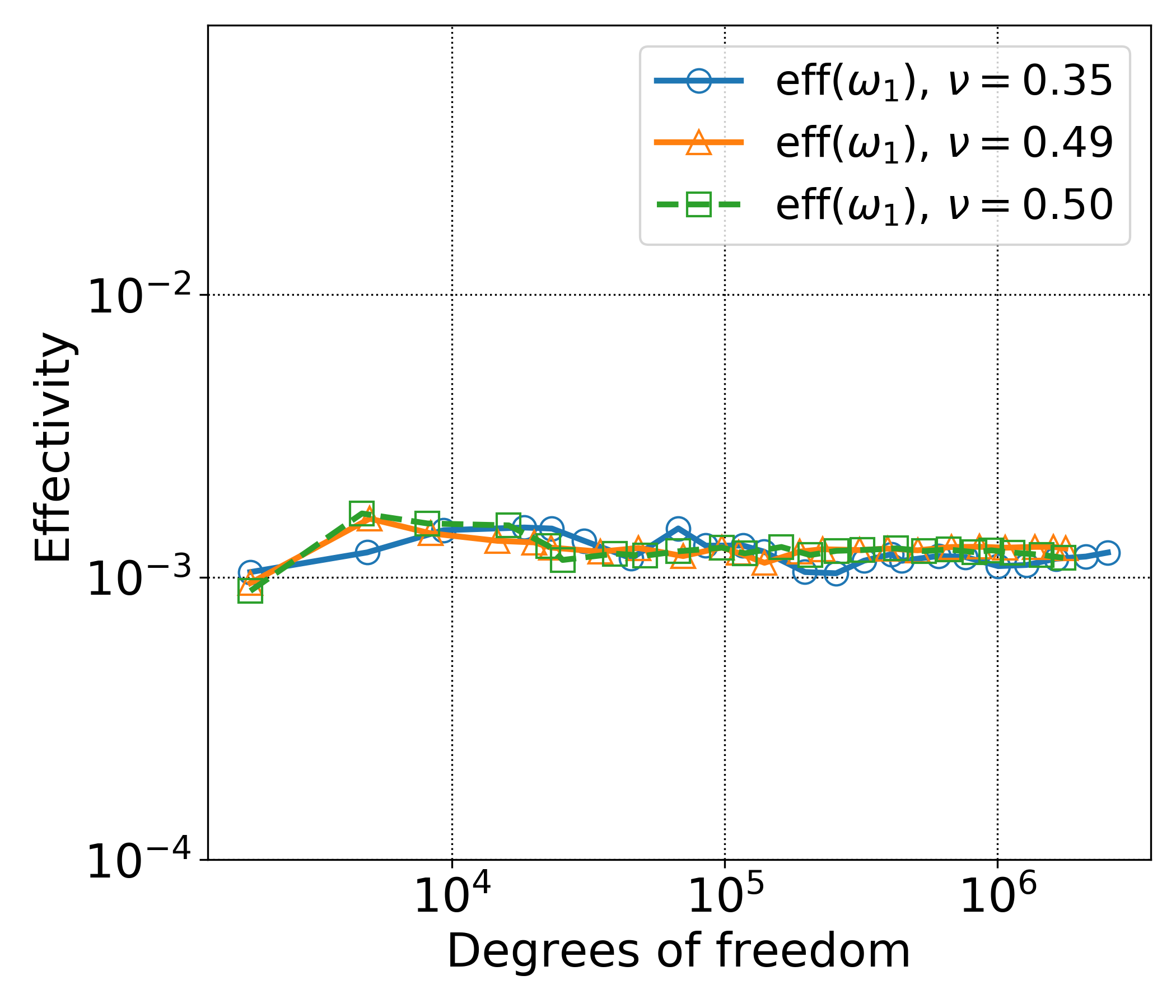}
			\end{minipage}
			\caption{Test 1. Error curves of the different selected values of $\nu$ in the three dimensional L-shaped domain (top),  estimator values curve compared with $\mathcal{O}(N^{-2/3})$ (bottom left), and effectivity of the estimator for different values of $\nu$ (bottom right).}
			\label{fig:lshape-3d-error}
		\end{figure}
		
		
		\begin{figure}
			\centering
			\begin{minipage}{0.49\linewidth}\centering
				\includegraphics[scale=0.10,trim=22cm 4.2cm 21cm 2.4cm,clip]{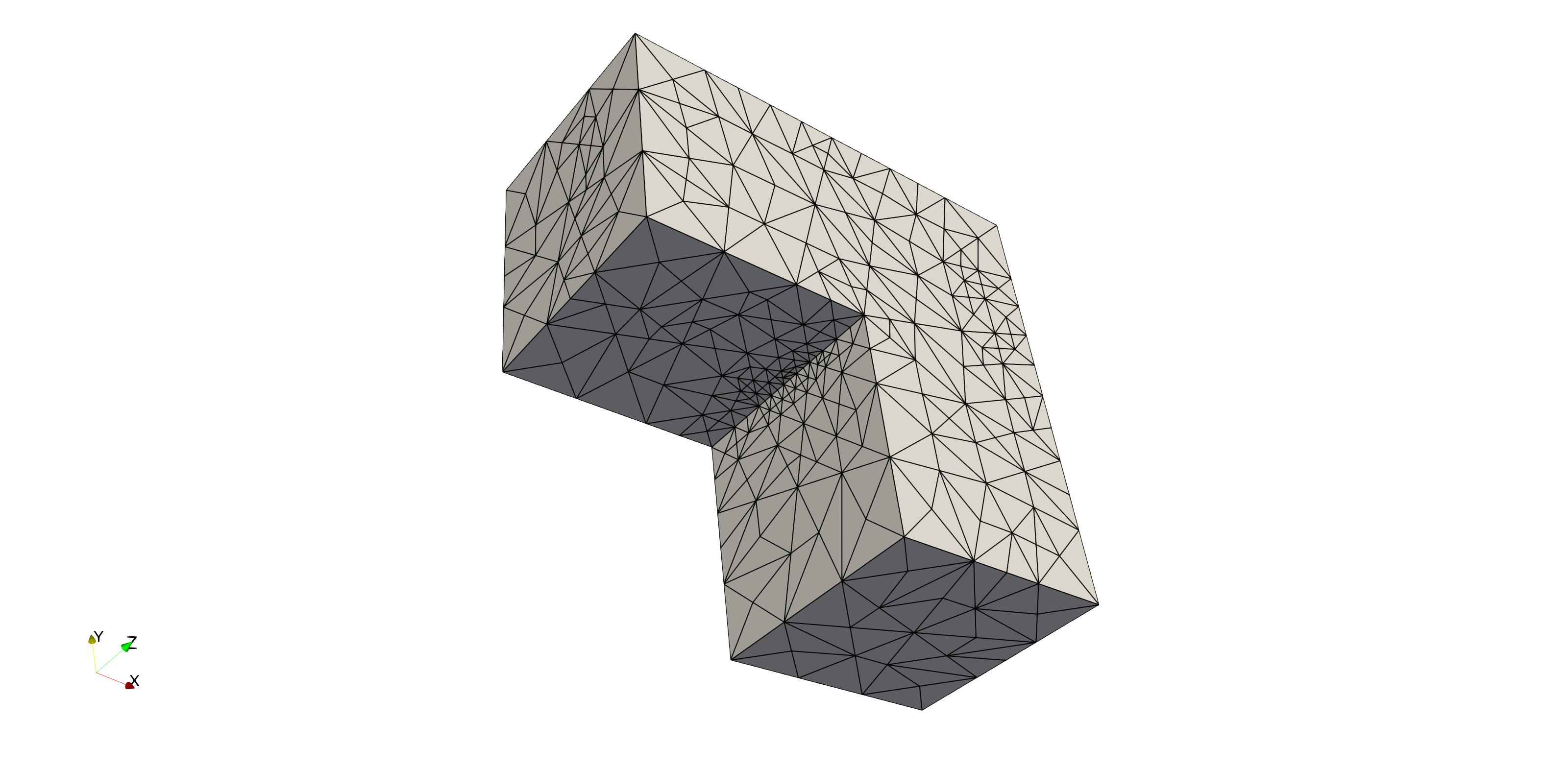}
			\end{minipage}
			\begin{minipage}{0.49\linewidth}\centering
				\includegraphics[scale=0.10,trim=22cm 4.2cm 21cm 2.4cm,clip]{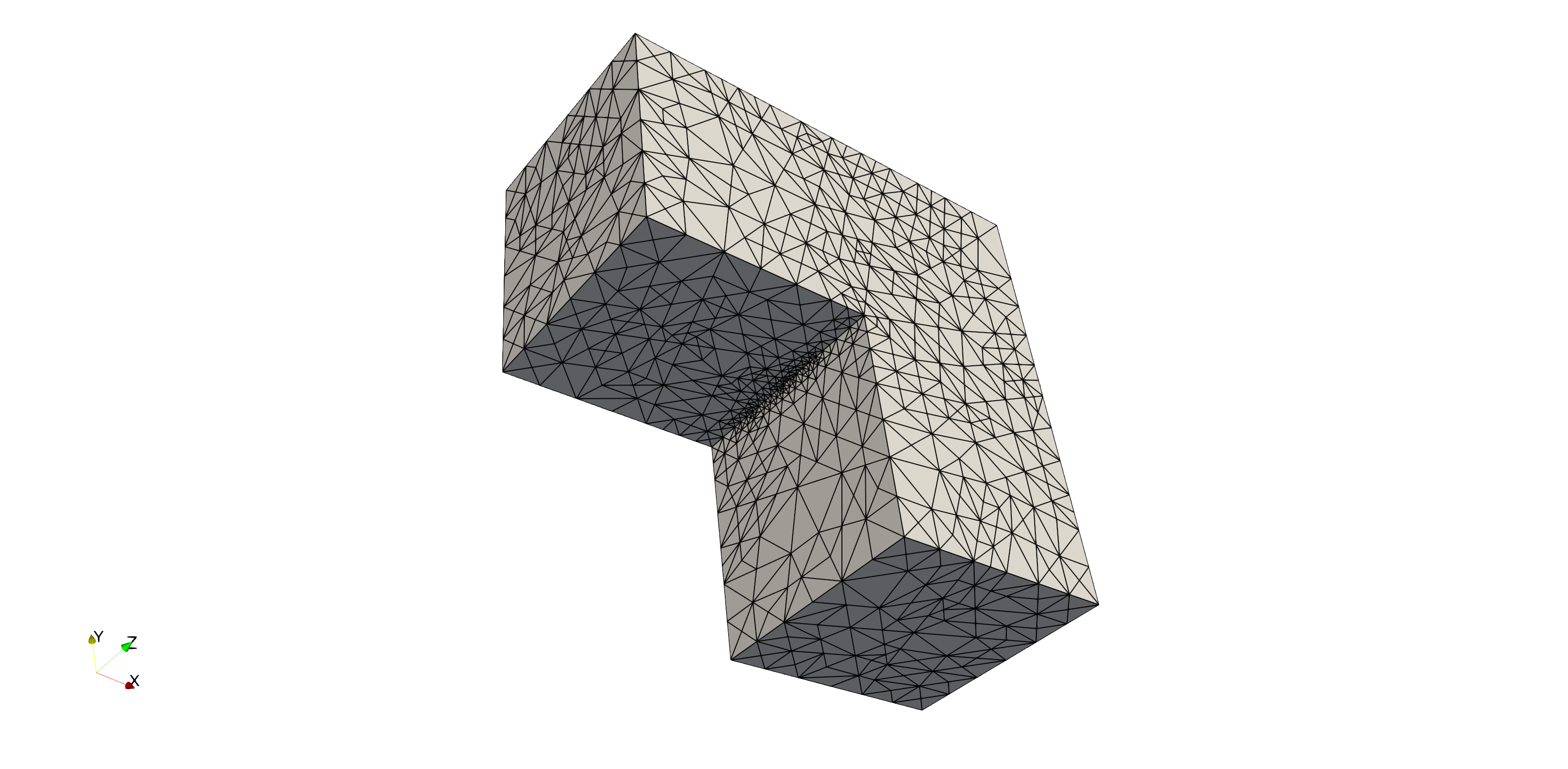}
			\end{minipage}
			\begin{minipage}{0.49\linewidth}\centering
				\includegraphics[scale=0.10,trim=22cm 4.2cm 21cm 2.4cm,clip]{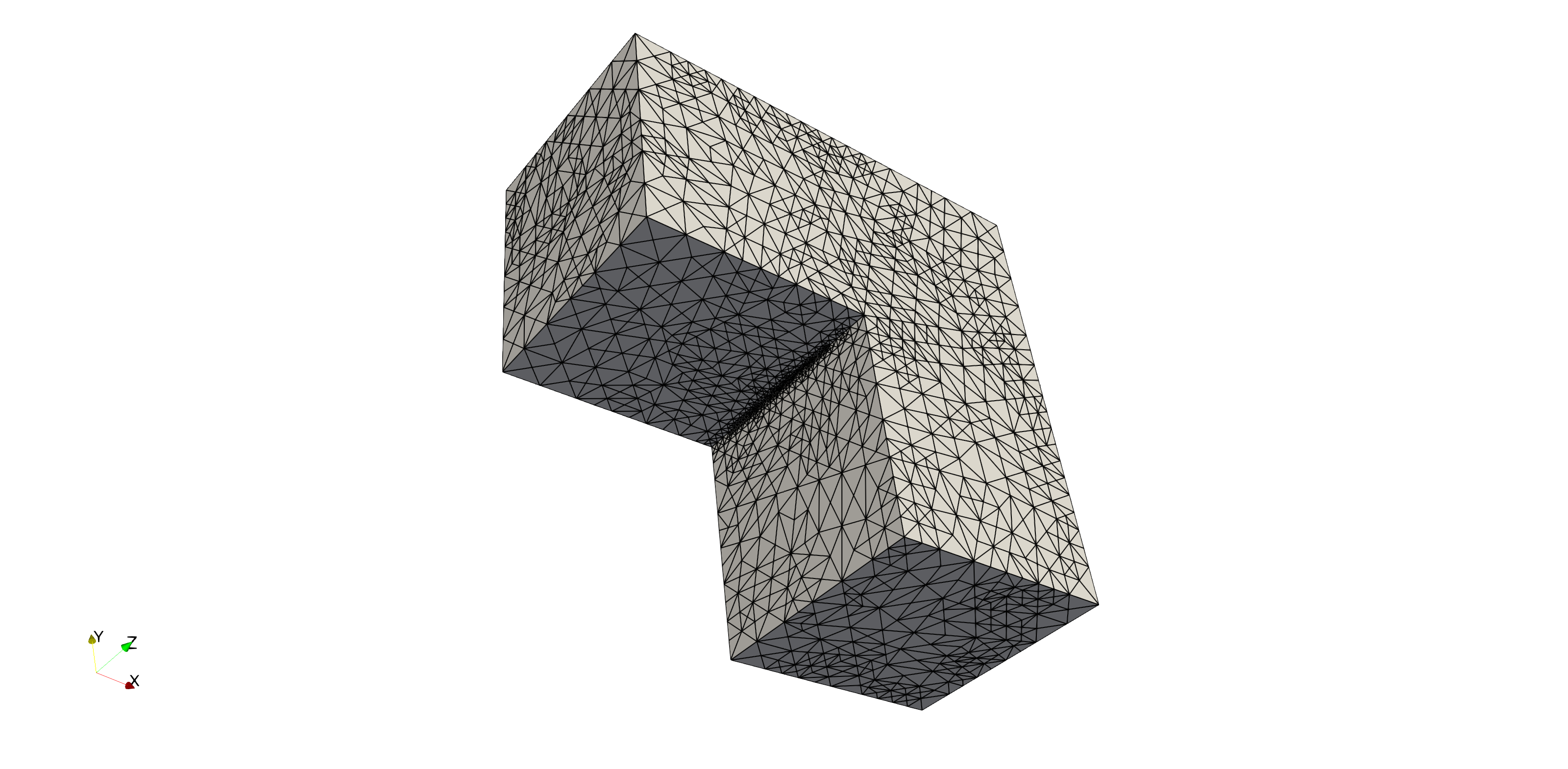}
			\end{minipage}
			\begin{minipage}{0.49\linewidth}\centering
				\includegraphics[scale=0.10,trim=22cm 4.2cm 21cm 2.4cm,clip]{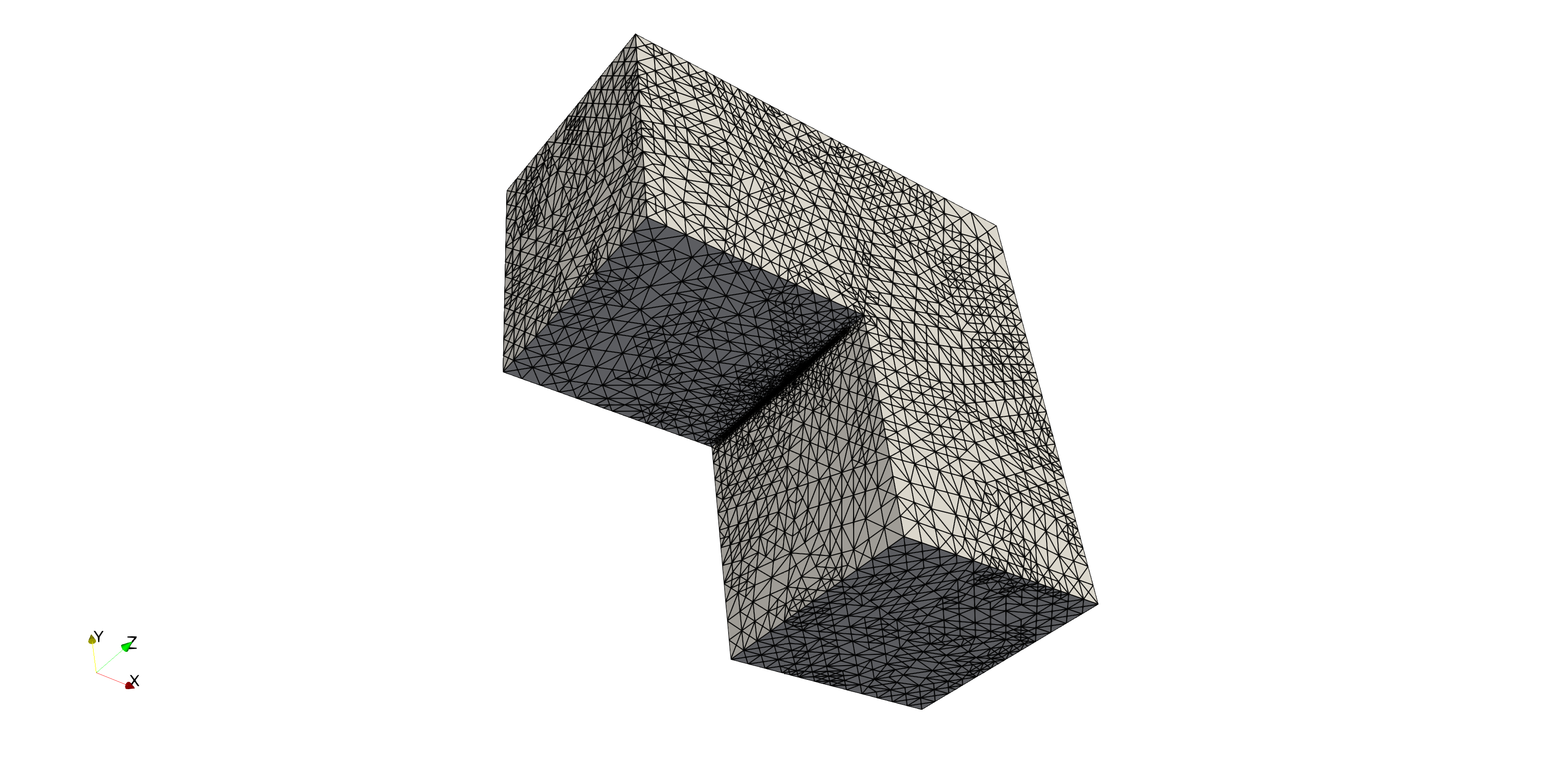}
			\end{minipage}
			\caption{Test 1. Mesh adaptive refinement at the eighth, twelfth, seventeenth, and final step when $\nu=0.35$.}
			\label{fig:l-3d-nu035-mallas}
		\end{figure}
		
		\begin{figure}
			\centering
			\begin{minipage}{0.49\linewidth}\centering
				\includegraphics[scale=0.10,trim=22cm 4.2cm 21cm 2.4cm,clip]{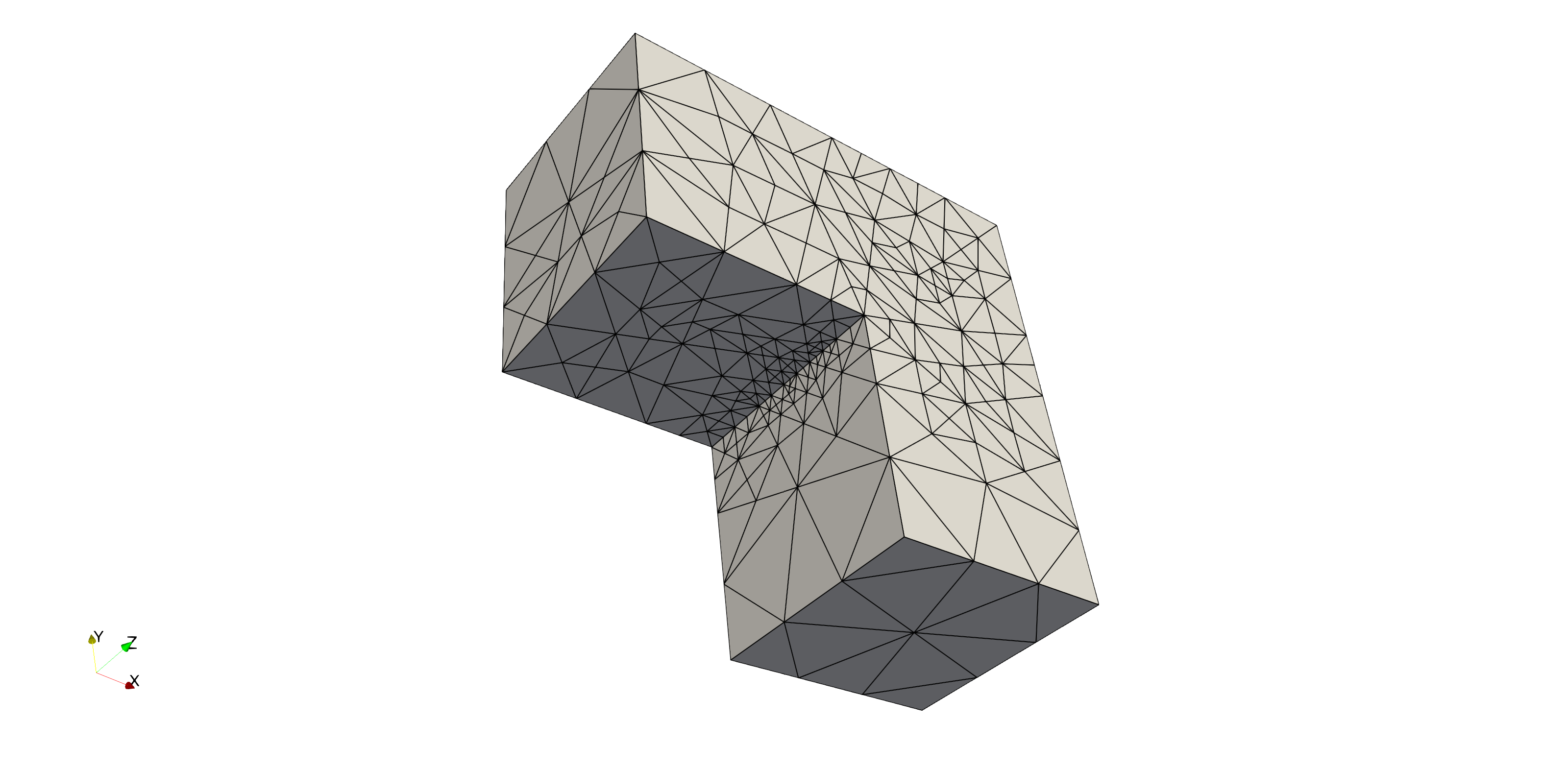}
			\end{minipage}
			\begin{minipage}{0.49\linewidth}\centering
				\includegraphics[scale=0.10,trim=22cm 4.2cm 21cm 2.4cm,clip]{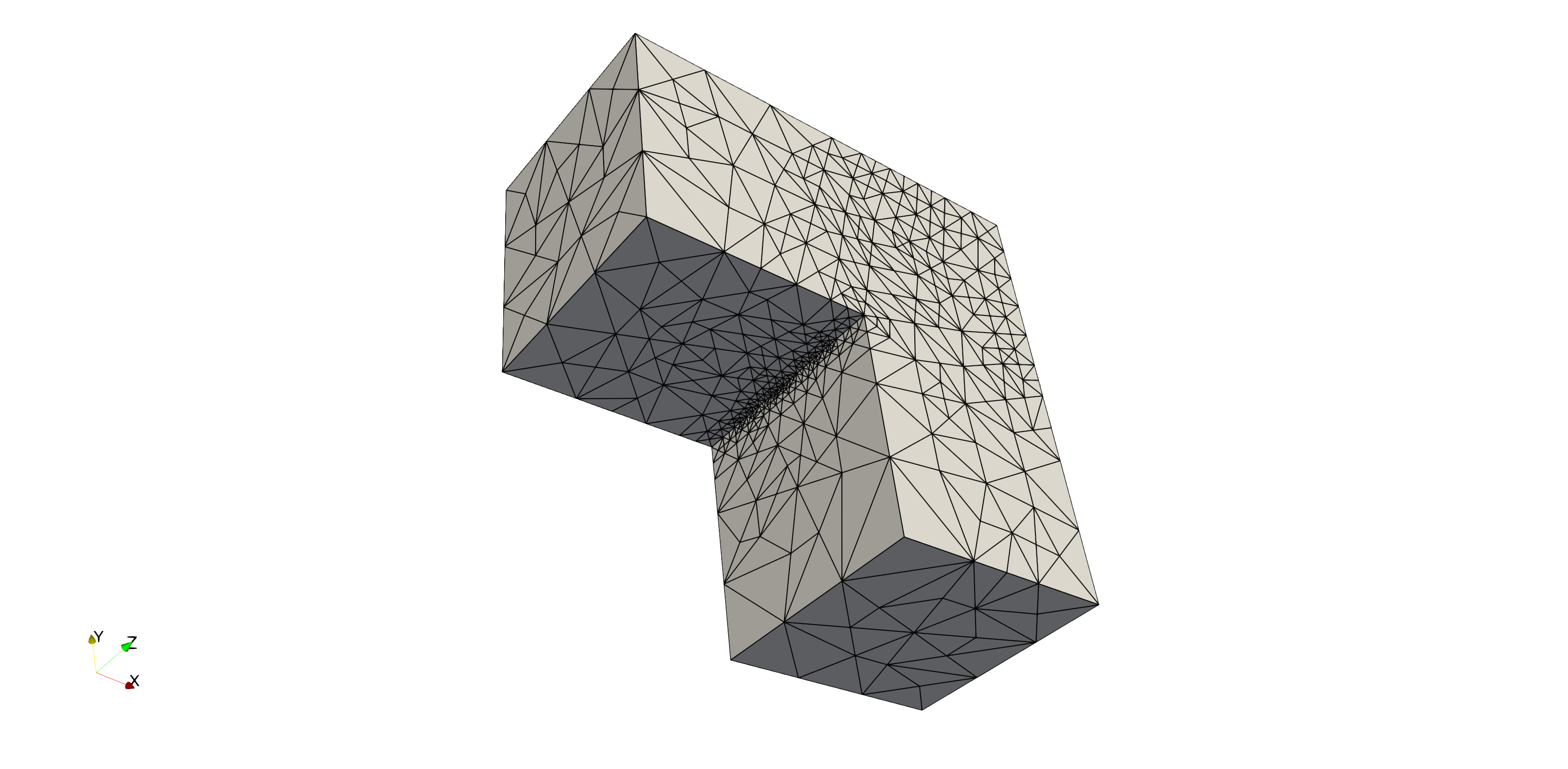}
			\end{minipage}
			\begin{minipage}{0.49\linewidth}\centering
				\includegraphics[scale=0.10,trim=22cm 4.2cm 21cm 2.4cm,clip]{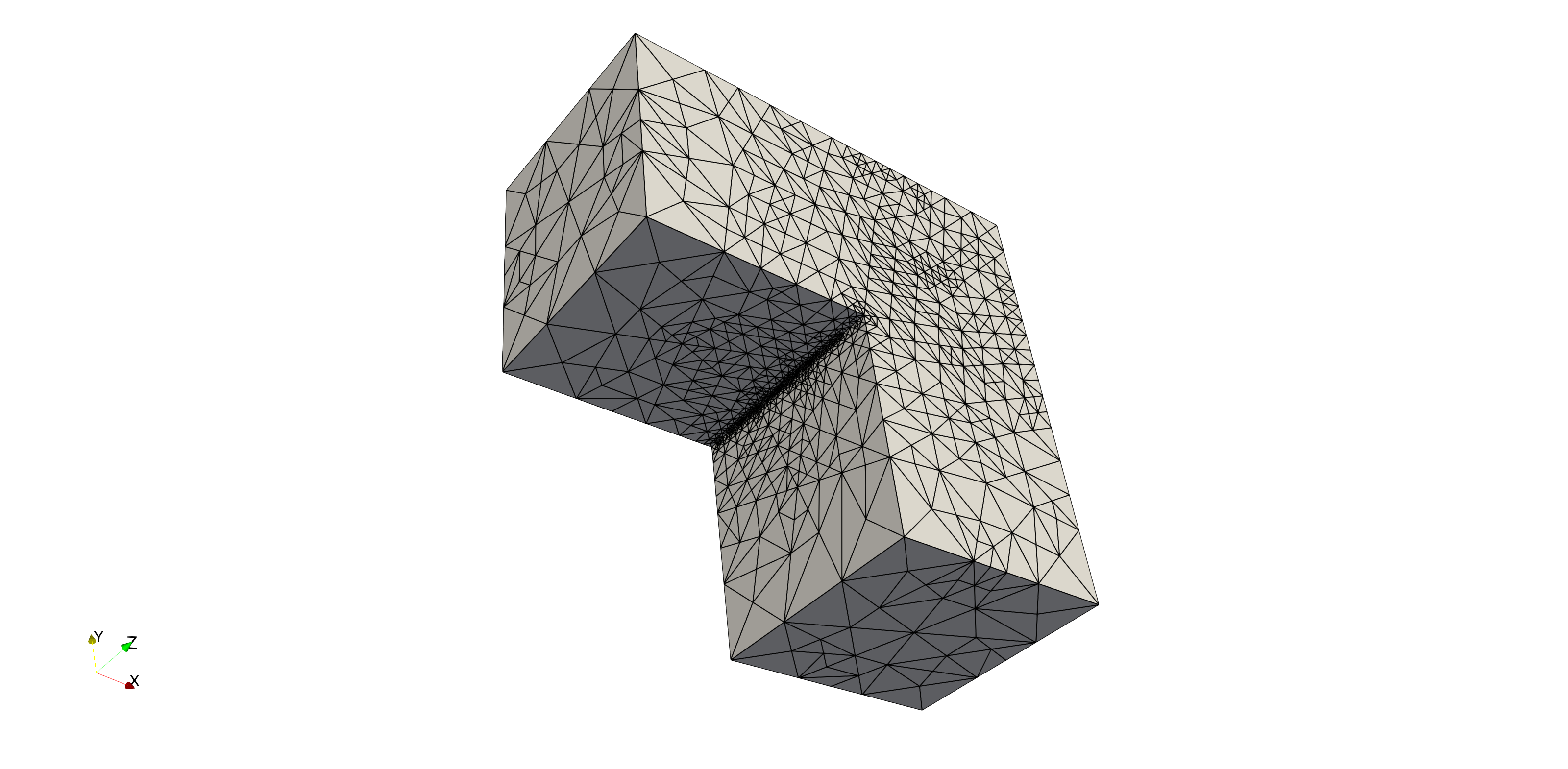}
			\end{minipage}
			\begin{minipage}{0.49\linewidth}\centering
				\includegraphics[scale=0.10,trim=22cm 4.2cm 21cm 2.4cm,clip]{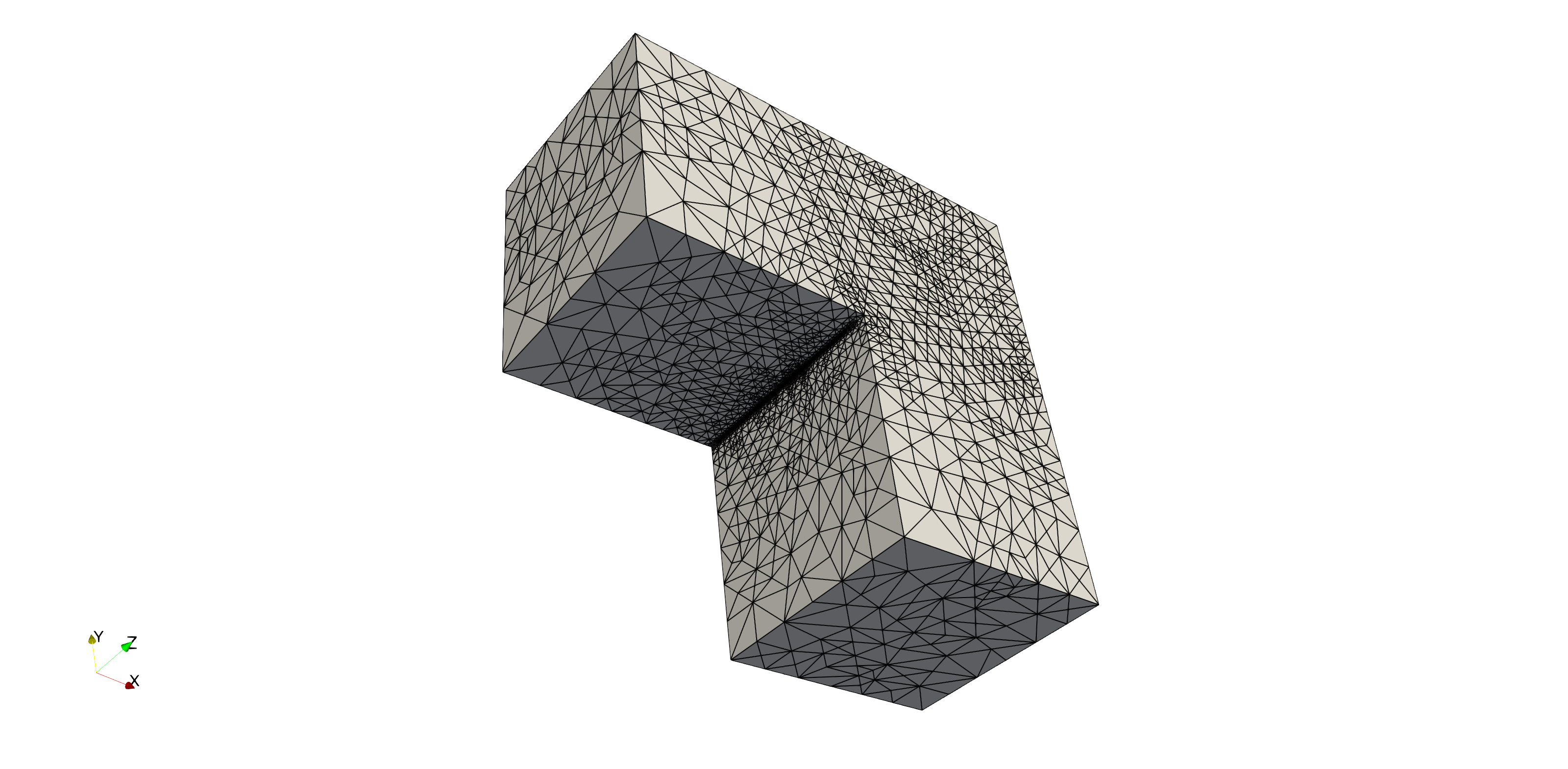}
			\end{minipage}
			\caption{Test 1. Mesh adaptive refinement at the eighth, twelfth, seventeenth, and final step when $\nu=0.5$.}
			\label{fig:l-3d-nu050-mallas}
		\end{figure}
\subsection{Test 2. The L-shaped domain} We now consider the classic L-shape domain occupying the region $\Omega:=(-1,1)\times(-1,1) \backslash \bigg((-1,0)\times (-1,0)\bigg)$, with initial shape given in Figure \ref{fig:lshapeinitialmesh}. The goal of this experiment is to confirm the robustness of our mixed adaptive schemes in lower dimensions and non-convex domains. In this case, we replace tetrahedrons with triangles and edges instead of faces. Hence, in order to define the two-dimensional estimator, we note that in \eqref{eq:local_eta} we have that $\curl$ corresponds to $\rot$, and consider $\btau\boldsymbol{t}$ instead of $\btau\times\bn$, where $\boldsymbol{t}:=(n_2,-n_1)$ is a fixed unit tangent vector to the edge $e$.
	The extrapolated values for this experiment have been obtained through sufficiently fine meshing and least squares fitting. In Table \ref{table: l-2d-tabla-de-extrapolados} we can see the calculated values, which are in good agreement with those obtained in \cite{inzunza2021displacementpseudostress}.
	\begin{table}[H]
		\begin{center}
			\begin{tabular}{c|c}
				$\nu$ & $\omega_1$\\\midrule
				0.35 & 2.37877\\
				0.49 & 3.26873 \\
				0.5 & 3.27271 \\
			\end{tabular}
		\end{center}
		\caption{Test 2. Lowest computed eigenvalues using highly refined meshes and least square fitting in the two dimensional L-shaped domain.}
		\label{table: l-2d-tabla-de-extrapolados}
	\end{table}
	We begin by presenting in Figure \ref{fig:lshape-error} (top) a comparison between the errors calculated using uniform refinements and the adaptive scheme. The fit curves obtained using uniform mesh refinements are $-0.6$ ($\nu=0.35$), $-0.58$ ($\nu=0.49$), and $-0.57$ ($\nu=0.5$). Hence, the computed eigenfrequencies behaves like $\mathcal{O}(h^{1.2})$ ($\nu=0.35$), $\mathcal{O}(h^{1.16})$ ($\nu=0.49$) and $\mathcal{O}(h^{1.14})$ ($\nu=0.5$), being these the best rates to expect using this type of refinement. On the other hand, although the reliability of the estimator is not guaranteed for non-convex geometries, the fit lines in the adaptive refinements show a slope of $-1.0$ ($\nu=0.35$), and $-1.02$ ($\nu=0.49$ and $\nu=0.50$).  This implies that the computed eigenfrequencies satisfies $\vert w_1-w_{1h}\vert=CN^{-1.0}=Ch^{2.0}$ ($\nu=0.35$), and $\vert w_1-w_{1h}\vert=CN^{-1.02}=Ch^{2.04}$, showing that our estimator is able to recover the optimal order of convergence and it is not affected by the singularity in $(0,0)$. Note also that Figure \ref{fig:lshape-error} (bottom left and bottom right) shows that our estimator behaves as $\mathcal{O}(N^{-1})$ in all cases, therefore, as predicted theoretically, the effectivity indexes remain bounded. In Figure  \ref{fig:lshape-mesh-solution_nu050} we depict intermediate meshes when running our adaptive algorithm in the limit case.  We observe that our scheme is not affected and the adaptive scheme is capable of detect and refine near the singularity.

%
	
	\begin{figure}
		\centering
		\includegraphics[scale=0.08,trim= 16cm 0 16cm 0, clip]{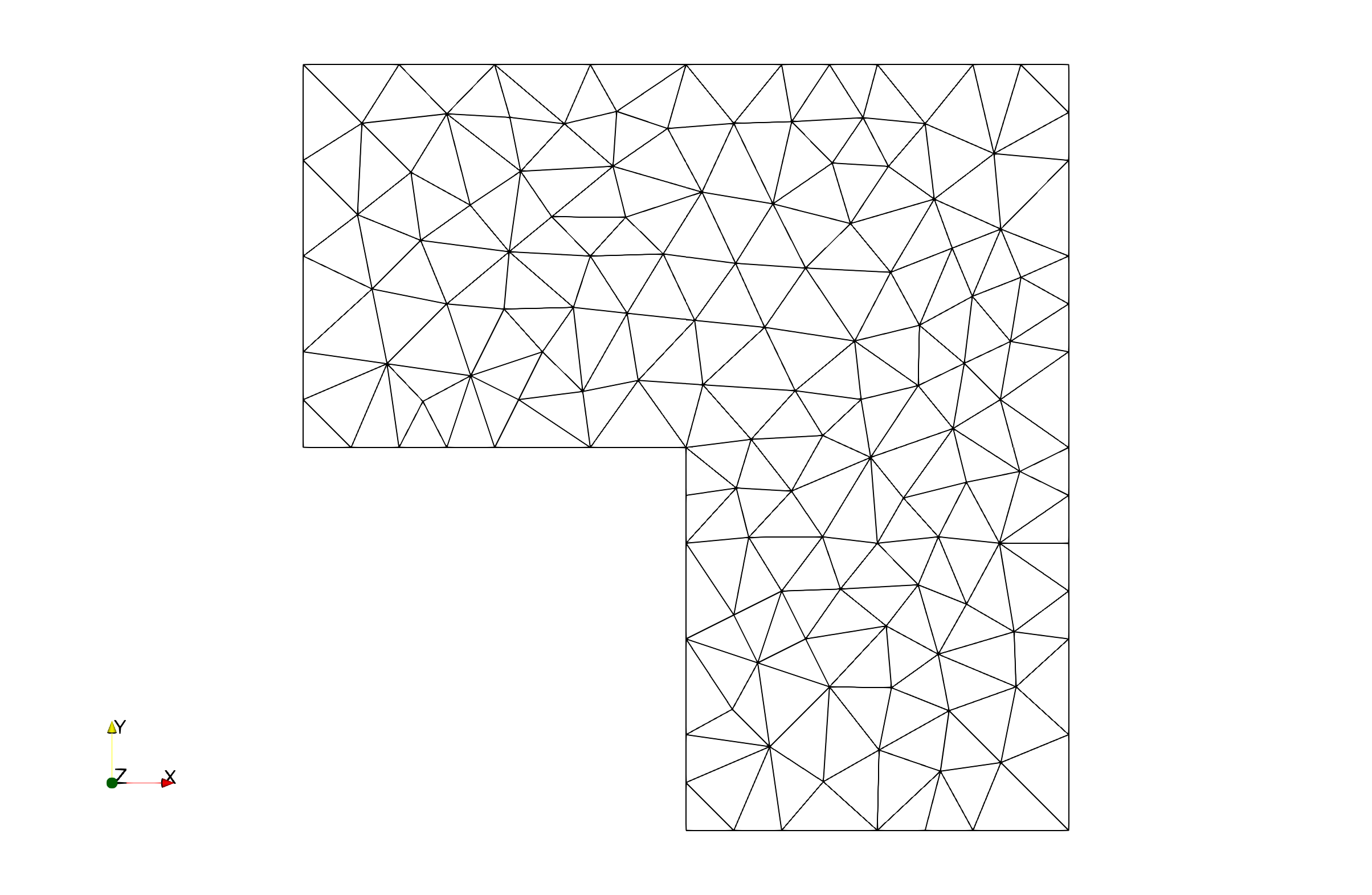}
		\caption{Test 2. Initial mesh on the L-shaped domain.}
		\label{fig:lshapeinitialmesh}
	\end{figure}
	\begin{figure}
		\centering
		\begin{minipage}{\linewidth}\centering
			\includegraphics[scale=0.35]{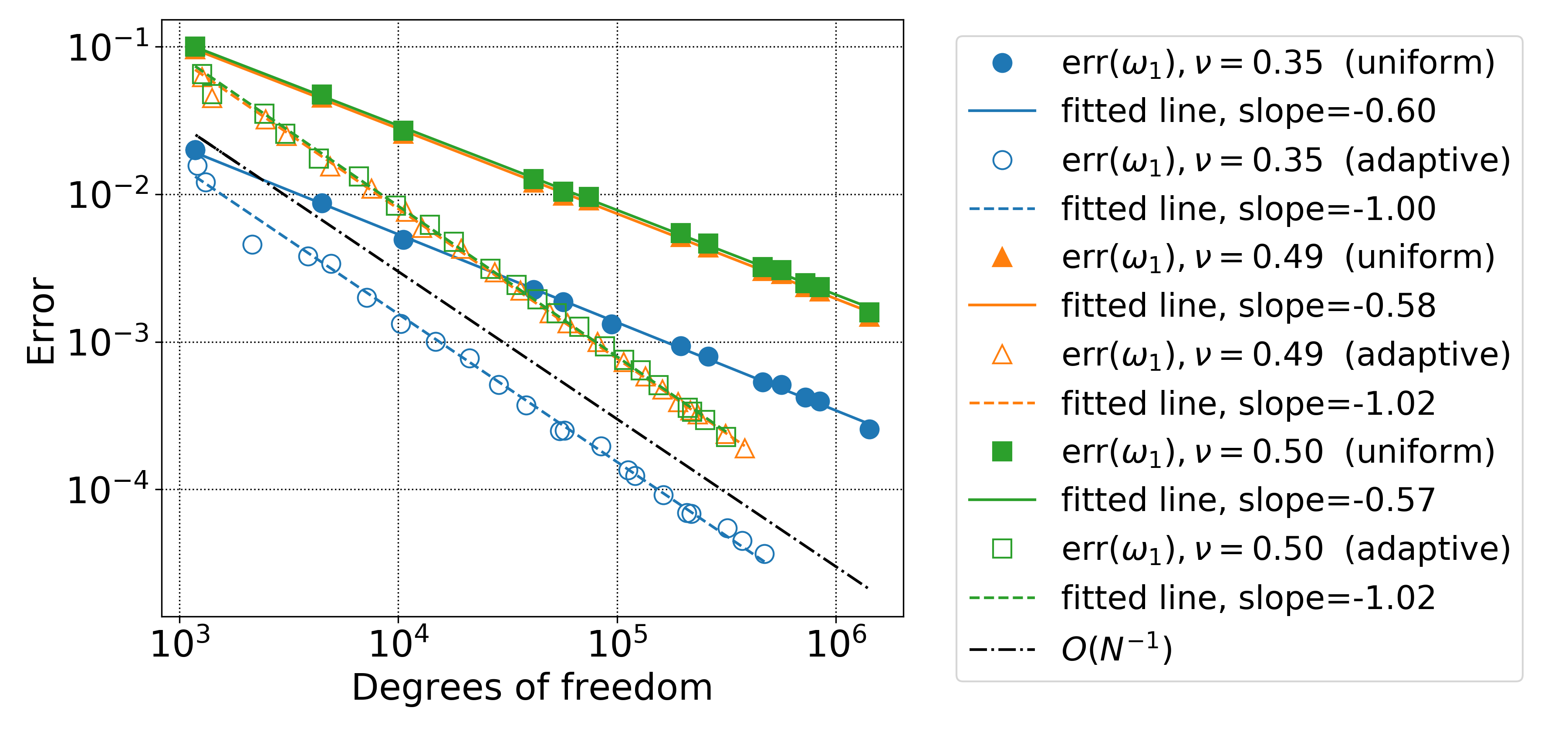}
		\end{minipage}\\
		\begin{minipage}{0.49\linewidth}\centering
			\includegraphics[scale=0.35]{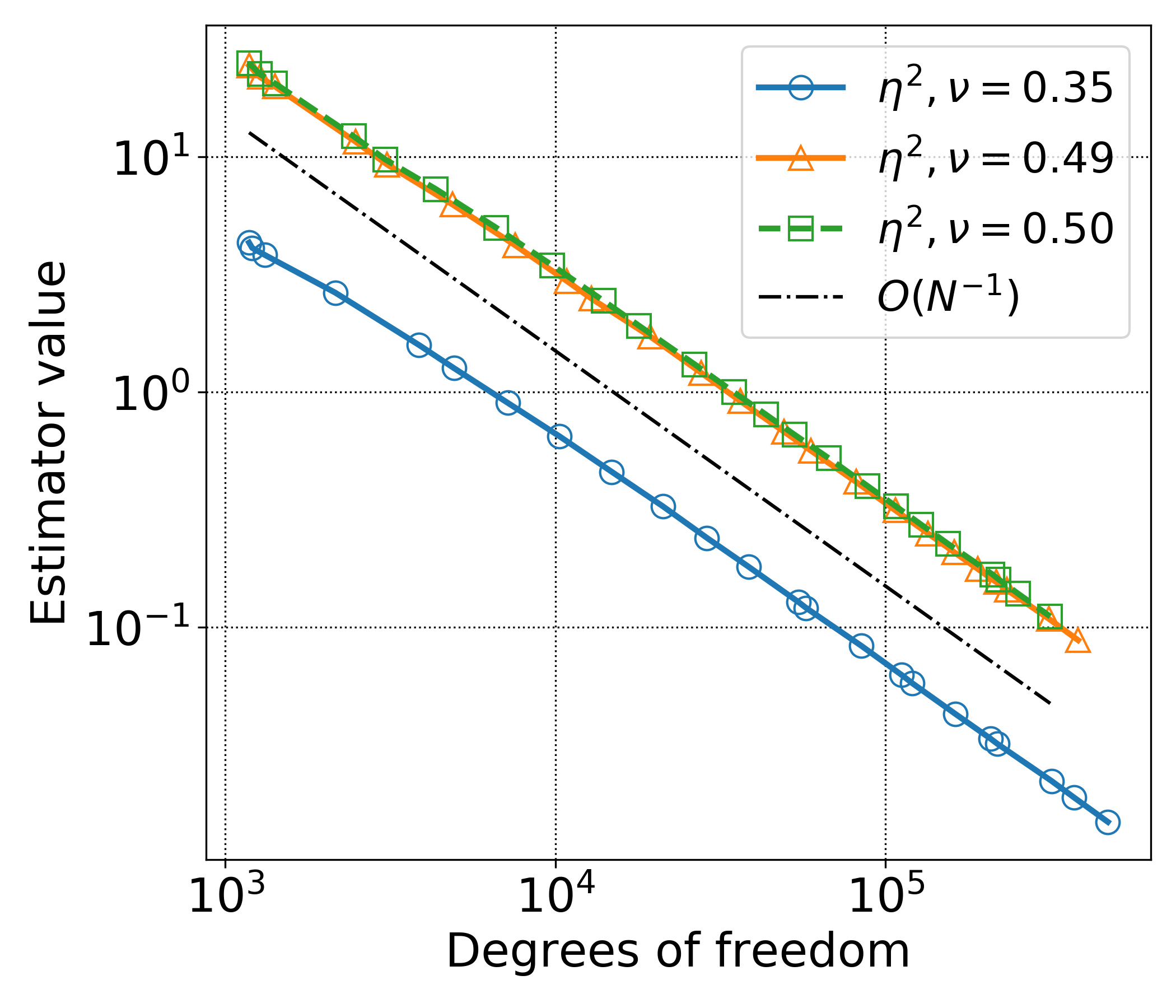}
		\end{minipage}
		\begin{minipage}{0.49\linewidth}\centering
			\includegraphics[scale=0.35]{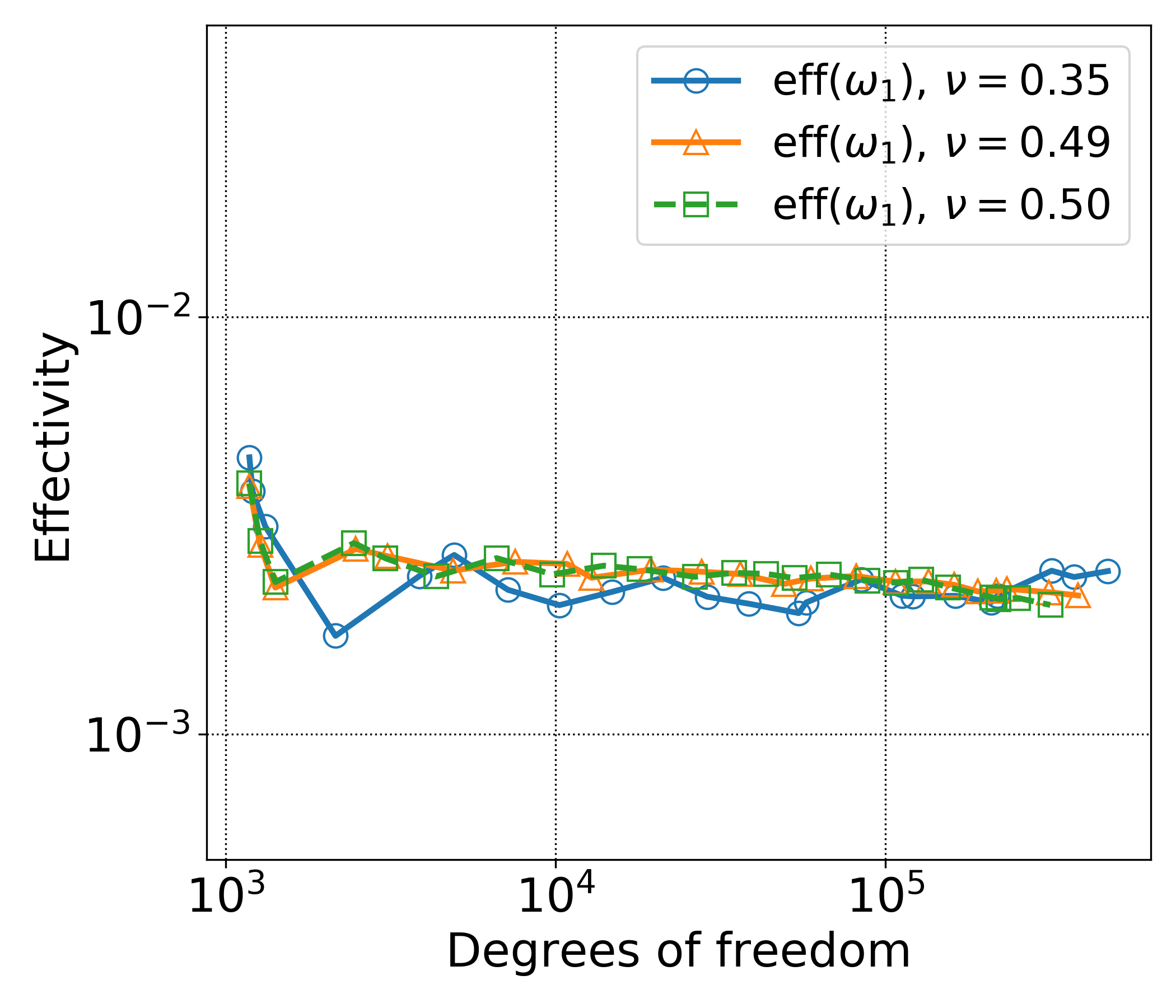}
		\end{minipage}
		\caption{Test 2. Error curves of the different selected values of $\nu$ in the two dimensional L-shaped domain (top),  estimator values curve compared with $\mathcal{O}(N^{-1})$ (bottom left), and effectivity of the estimator for different values of $\nu$ (bottom right). }
		\label{fig:lshape-error}
	\end{figure}
		
		\begin{figure}
			\centering
			\begin{minipage}{0.48\linewidth}
				\includegraphics[scale=0.085,trim= 16cm 2.5cm 16cm 2.5cm, clip]{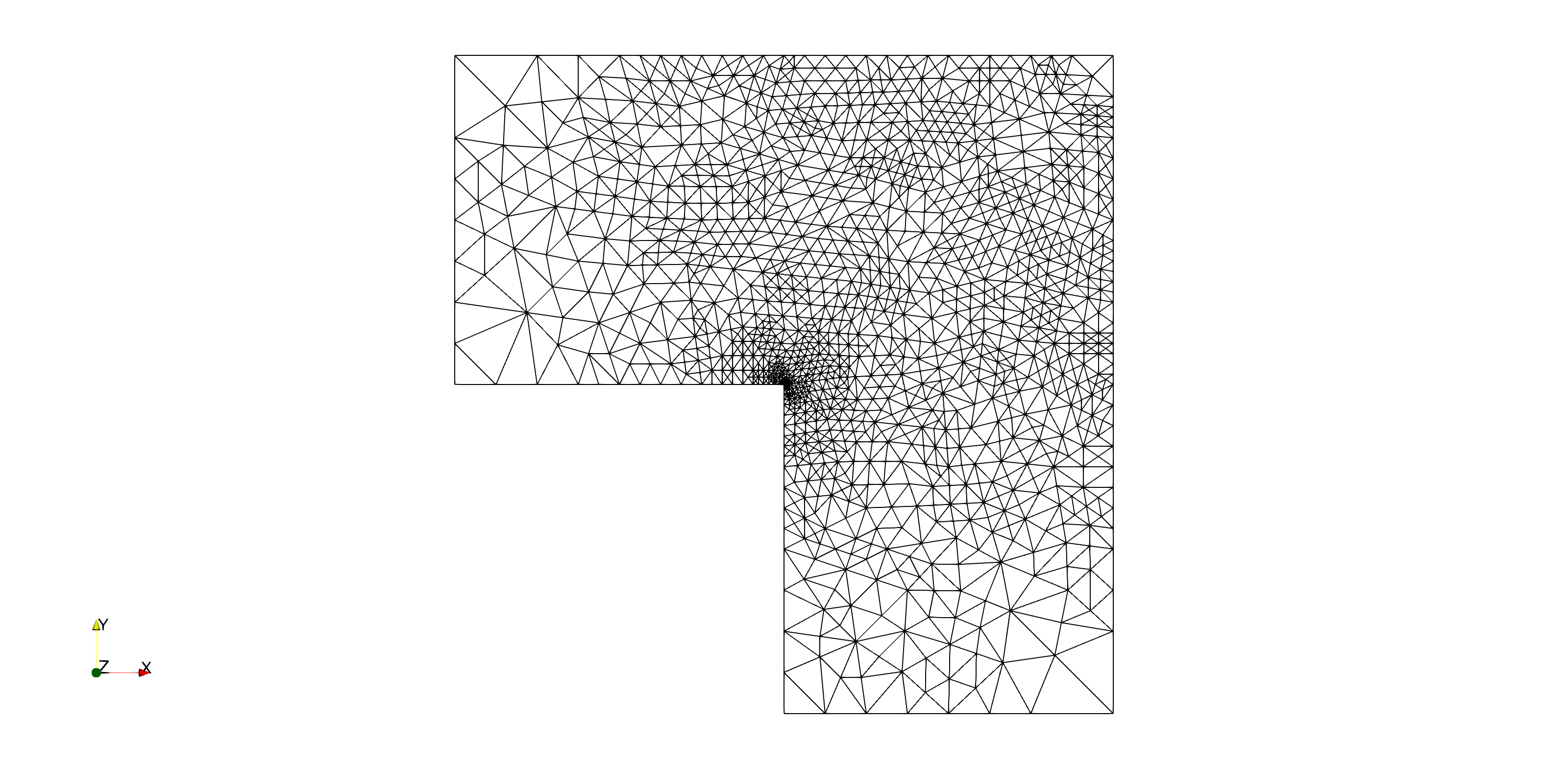}
			\end{minipage}
			\begin{minipage}{0.48\linewidth}
				\includegraphics[scale=0.085,trim= 16cm 2.5cm 16cm 2.5cm,, clip]{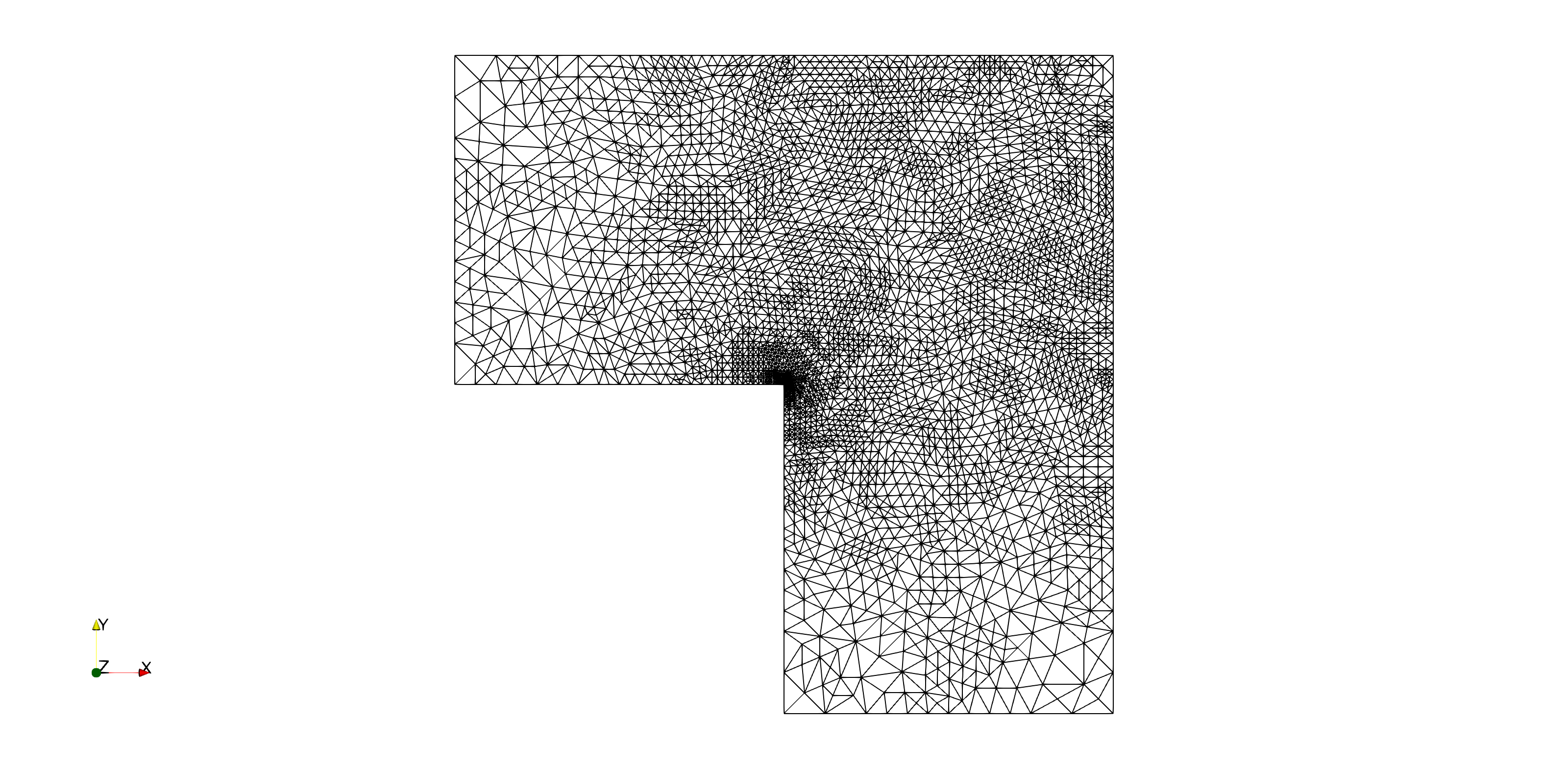}
			\end{minipage}\\
			\begin{minipage}{0.48\linewidth}
				\includegraphics[scale=0.085,trim= 16cm 2.5cm 16cm 2.5cm,, clip]{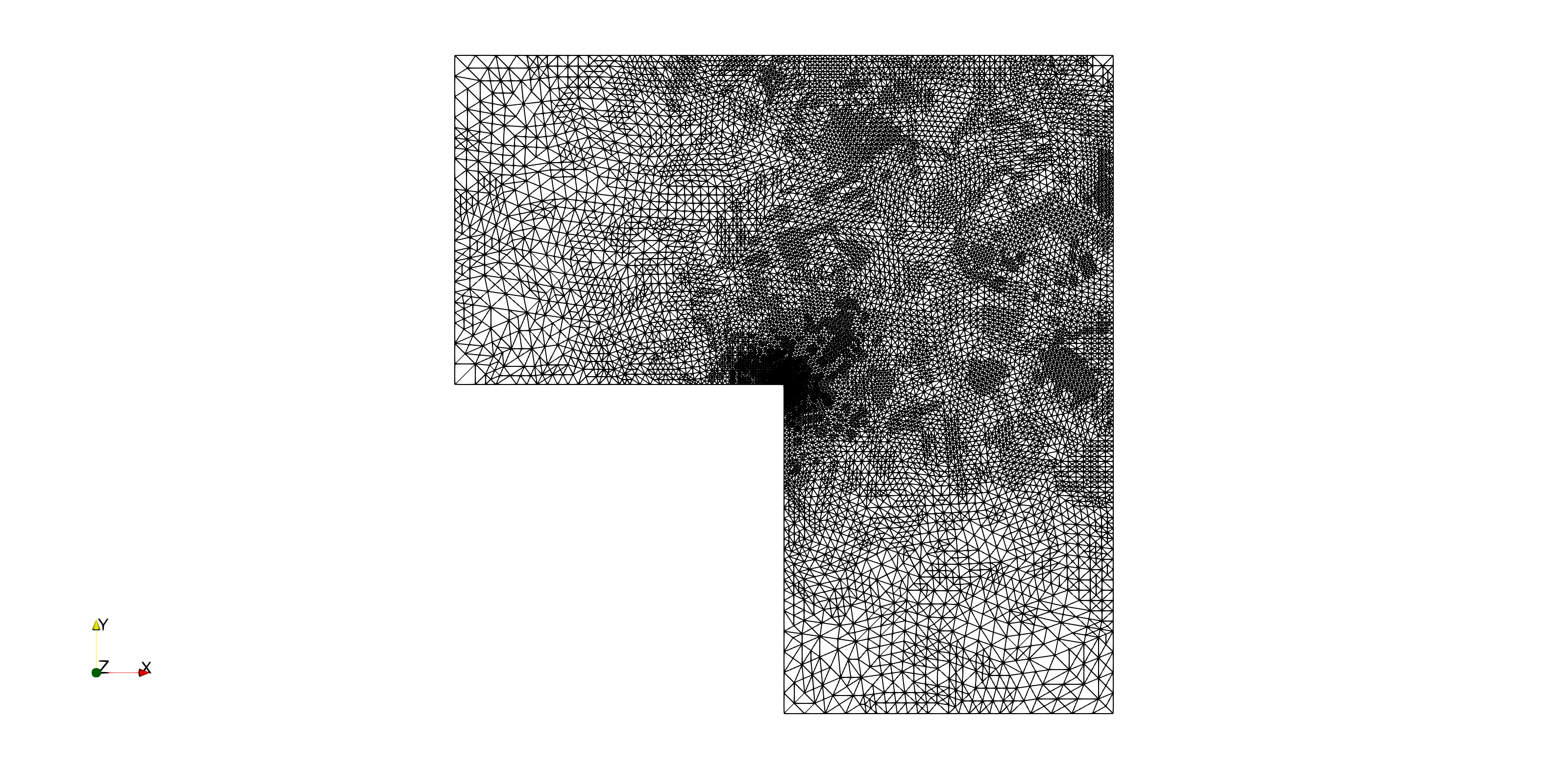}
			\end{minipage}
			\begin{minipage}{0.48\linewidth}
				\includegraphics[scale=0.085,trim= 16cm 2.5cm 16cm 2.5cm,, clip]{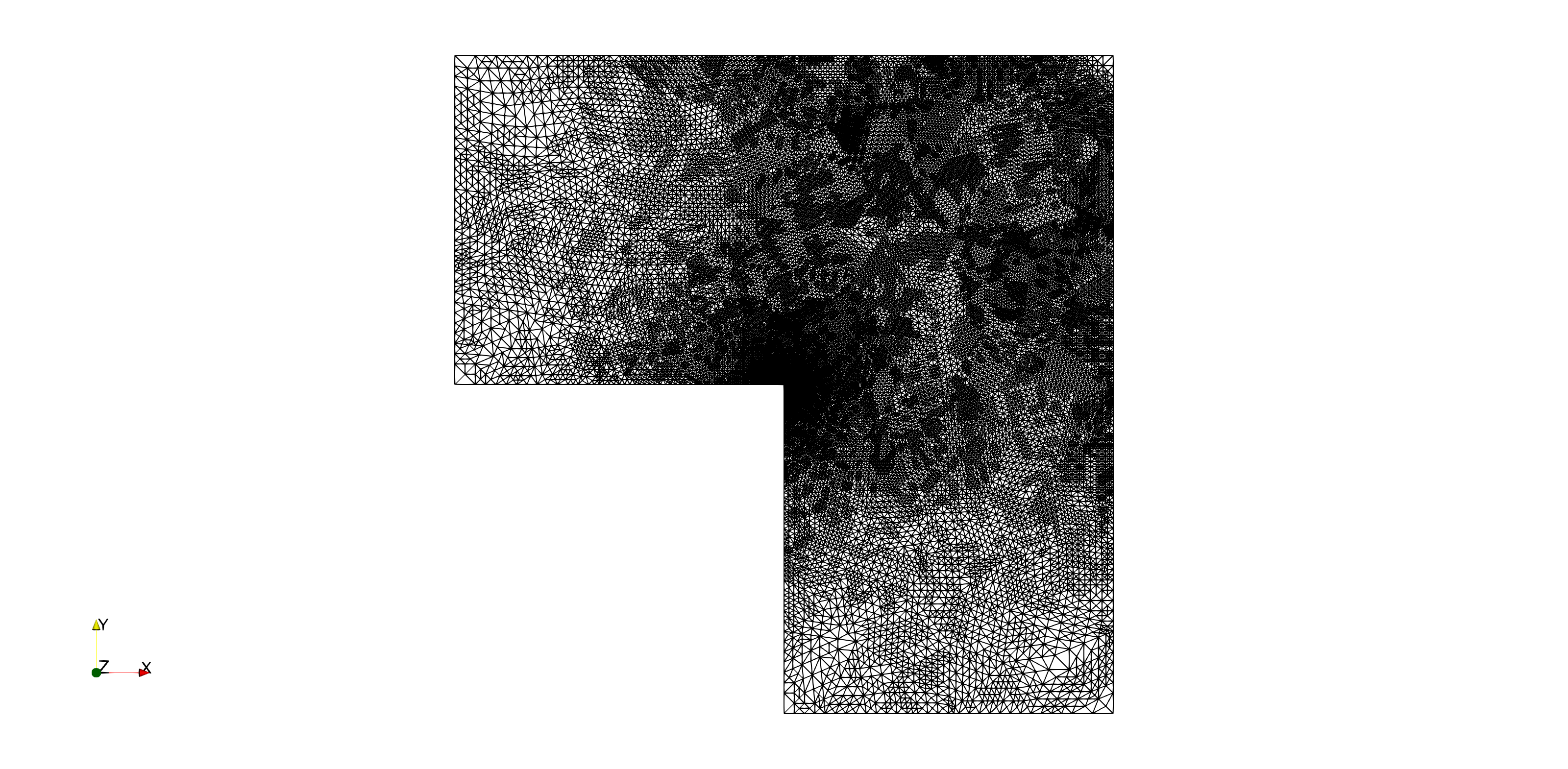}
			\end{minipage}
			\caption{Test 2. Mesh adaptive refinement at the eighth, twelfth, seventeenth, and final step when $\nu=0.5$.}
			\label{fig:lshape-mesh-solution_nu050}
		\end{figure}

\bibliographystyle{siam}
\footnotesize
\bibliography{bib_LOQ}

\begin{thebibliography}{10}

\bibitem{MR1946986}
{\sc A.~Alonso, A.~Dello~Russo, C.~Otero-Souto, C.~Padra, and
  R.~Rodr\'{\i}guez}, {\em An adaptive finite element scheme to solve
  fluid-structure vibration problems on non-matching grids}, vol.~4, 2001,
  pp.~67--78.
\newblock Second AMIF International Conference (Il Ciocco, 2000).

\bibitem{MR2293249}
{\sc T.~P. Barrios, G.~N. Gatica, M.~Gonz\'{a}lez, and N.~Heuer}, {\em A
  residual based a posteriori error estimator for an augmented mixed finite
  element method in linear elasticity}, M2AN Math. Model. Numer. Anal., 40
  (2006), pp.~843--869 (2007).

\bibitem{MR4279087}
{\sc F.~Bertrand, D.~Boffi, and R.~Ma}, {\em An adaptive finite element scheme
  for the {H}ellinger-{R}eissner elasticity mixed eigenvalue problem}, Comput.
  Methods Appl. Math., 21 (2021), pp.~501--512.

\bibitem{MR3097958}
{\sc D.~Boffi, F.~Brezzi, and M.~Fortin}, {\em Mixed finite element methods and
  applications}, vol.~44 of Springer Series in Computational Mathematics,
  Springer, Heidelberg, 2013.

\bibitem{MR3647956}
{\sc D.~Boffi, D.~Gallistl, F.~Gardini, and L.~Gastaldi}, {\em Optimal
  convergence of adaptive {FEM} for eigenvalue clusters in mixed form}, Math.
  Comp., 86 (2017), pp.~2213--2237.

\bibitem{MR3712172}
{\sc D.~Boffi, L.~Gastaldi, R.~Rodr\'{\i}guez, and I.~\v{S}ebestov\'{a}}, {\em
  Residual-based {\it a posteriori} error estimation for the {M}axwell's
  eigenvalue problem}, IMA J. Numer. Anal., 37 (2017), pp.~1710--1732.

\bibitem{MR3918688}
\leavevmode\vrule height 2pt depth -1.6pt width 23pt, {\em A posteriori error
  estimates for {M}axwell's eigenvalue problem}, J. Sci. Comput., 78 (2019),
  pp.~1250--1271.

\bibitem{MR3452773}
{\sc C.~Carstensen and J.~Gedicke}, {\em Robust residual-based a posteriori
  {A}rnold-{W}inther mixed finite element analysis in elasticity}, Comput.
  Methods Appl. Mech. Engrg., 300 (2016), pp.~245--264.

\bibitem{MR3790080}
{\sc C.~Carstensen and F.~Hellwig}, {\em Optimal convergence rates for adaptive
  lowest-order discontinuous {P}etrov-{G}alerkin schemes}, SIAM J. Numer.
  Anal., 56 (2018), pp.~1091--1111.

\bibitem{MR2970742}
{\sc C.~Carstensen and H.~Rabus}, {\em The adaptive nonconforming {FEM} for the
  pure displacement problem in linear elasticity is optimal and robust}, SIAM
  J. Numer. Anal., 50 (2012), pp.~1264--1283.

\bibitem{MR2754580}
{\sc H.~Chen, S.~Jia, and H.~Xie}, {\em Postprocessing and higher order
  convergence for the mixed finite element approximations of the eigenvalue
  problem}, Appl. Numer. Math., 61 (2011), pp.~615--629.

\bibitem{MR0520174}
{\sc P.~G. Ciarlet}, {\em The finite element method for elliptic problems},
  North-Holland Publishing Co., Amsterdam-New York-Oxford, 1978.
\newblock Studies in Mathematics and its Applications, Vol. 4.

\bibitem{MR961439}
{\sc M.~Dauge}, {\em Elliptic boundary value problems on corner domains},
  vol.~1341 of Lecture Notes in Mathematics, Springer-Verlag, Berlin, 1988.
\newblock Smoothness and asymptotics of solutions.

\bibitem{MR1722056}
{\sc R.~G. Dur\'{a}n, L.~Gastaldi, and C.~Padra}, {\em A posteriori error
  estimators for mixed approximations of eigenvalue problems}, Math. Models
  Methods Appl. Sci., 9 (1999), pp.~1165--1178.

\bibitem{MR3453481}
{\sc G.~N. Gatica, L.~F. Gatica, and F.~A. Sequeira}, {\em A priori and a
  posteriori error analyses of a pseudostress-based mixed formulation for
  linear elasticity}, Comput. Math. Appl., 71 (2016), pp.~585--614.

\bibitem{MR3093586}
{\sc G.~N. Gatica, A.~M\'{a}rquez, and W.~Rudolph}, {\em A priori and a
  posteriori error analyses of augmented twofold saddle point formulations for
  nonlinear elasticity problems}, Comput. Methods Appl. Mech. Engrg., 264
  (2013), pp.~23--48.

\bibitem{MR840970}
{\sc P.~Grisvard}, {\em Probl\`emes aux limites dans les polygones. {M}ode
  d'emploi}, EDF Bull. Direction \'{E}tudes Rech. S\'{e}r. C Math. Inform.,
  (1986), pp.~3, 21--59.

\bibitem{MR2009375}
{\sc R.~Hiptmair}, {\em Finite elements in computational electromagnetism},
  Acta Numer., 11 (2002), pp.~237--339.

\bibitem{MR2220917}
{\sc P.~Houston, D.~Sch\"{o}tzau, and T.~P. Wihler}, {\em An {$hp$}-adaptive
  mixed discontinuous {G}alerkin {FEM} for nearly incompressible linear
  elasticity}, Comput. Methods Appl. Mech. Engrg., 195 (2006), pp.~3224--3246.

\bibitem{inzunza2021displacementpseudostress}
{\sc D.~Inzunza, F.~Lepe, and G.~Rivera}, {\em Displacement-pseudostress
  formulation for the linear elasticity spectral problem: a priori analysis},
  https://arxiv.org/abs/2101.09828,  (2021).

\bibitem{MR3047040}
{\sc S.~Jia, H.~Chen, and H.~Xie}, {\em A posteriori error estimator for
  eigenvalue problems by mixed finite element method}, Sci. China Math., 56
  (2013), pp.~887--900.

\bibitem{MR3618064}
{\sc H.~P. Langtangen and A.~Logg}, {\em Solving {PDE}s in {P}ython}, vol.~3 of
  Simula SpringerBriefs on Computing, Springer, Cham, 2016.
\newblock The FEniCS tutorial I.

\bibitem{MR3962898}
{\sc F.~Lepe, S.~Meddahi, D.~Mora, and R.~Rodr\'{\i}guez}, {\em Mixed
  discontinuous {G}alerkin approximation of the elasticity eigenproblem},
  Numer. Math., 142 (2019), pp.~749--786.

\bibitem{MR4050542}
{\sc D.~Mora and G.~Rivera}, {\em {\it {A} priori} and {\it a posteriori} error
  estimates for a virtual element spectral analysis for the elasticity
  equations}, IMA J. Numer. Anal., 40 (2020), pp.~322--357.

\bibitem{MR1284252}
{\sc R.~Verf\"{u}rth}, {\em A posteriori error estimation and adaptive
  mesh-refinement techniques}, in Proceedings of the {F}ifth {I}nternational
  {C}ongress on {C}omputational and {A}pplied {M}athematics ({L}euven, 1992),
  vol.~50, 1994, pp.~67--83.

\bibitem{MR3059294}
{\sc R.~Verf\"{u}rth}, {\em A posteriori error estimation techniques for finite
  element methods}, Numerical Mathematics and Scientific Computation, Oxford
  University Press, Oxford, 2013.

\bibitem{verfuhrt1996}
{\sc R.~Verführt}, {\em A review of a posteriori error estimation and adaptive
  mesh-refinement techniques}, Advances in numerical mathematics, Wiley, 1996.

\end{thebibliography}
				
\end{document}